\documentclass[12pt, reqno]{amsart}
\usepackage{amsmath}
\usepackage{amssymb}
\usepackage{graphicx}
\usepackage{geometry}
\usepackage{subfigure}
\usepackage{amscd}
\usepackage[all]{xy}

\usepackage{marginnote}

\newtheorem{theorem}{Theorem}[section]
\newtheorem{corollary}[theorem]{Corollary}
\newtheorem{definition}[theorem]{Definition}
\newtheorem{lemma}[theorem]{Lemma}
\newtheorem{proposition}[theorem]{Proposition}
\newtheorem{conjecture}[theorem]{Conjecture}
\newtheorem{remark}{Remark}
\newtheorem{question}[theorem]{Question}

\newtheorem{exercise}[theorem]{Exercise}
\newtheorem{problem}[theorem]{Problem}

\newcommand{\Real}{\mathbb R}
\newcommand{\eps}{\varepsilon}
\newcommand{\lam}{\lambda}
\newcommand{\To}{\longrightarrow}

\newcommand{\norm}[1]{\left\Vert#1\right\Vert}

\newcommand{\Ws}[1]{\ensuremath{W^s(#1)}}
\newcommand{\Wu}[1]{\ensuremath{W^u(#1)}}
\newcommand{\Wc}[1]{\ensuremath{W^c(#1)}}

\newcommand{\diff}{\operatorname{Diff}}

\newcommand{\tto}{\Rightarrow}

\newcommand{\Wsl}[2][\epsilon]{\ensuremath{W_{#1}^s(#2)}}

\newcommand{\Wul}[2][\epsilon]{\ensuremath{W_{#1}^u(#2)}}

\newcommand{\Fc}{\ensuremath{\mathcal{W}^c}}
\newcommand{\Fs}{\ensuremath{\mathcal{W}^s}}
\newcommand{\Fu}{\ensuremath{\mathcal{W}^u}}
\newcommand{\Fcs}{\ensuremath{\mathcal{W}^{cs}}}
\newcommand{\Fcu}{\ensuremath{\mathcal{W}^{cu}}}
\newcommand{\F}{\ensuremath{\mathcal{F}}}
\newcommand{\oo}{\infty}
\newcommand{\T}{\mathbb T}
\newcommand{\Z}{\mathbb Z}
\newcommand{\R}{\mathbb R}
\newcommand{\A}{\mathcal A}

\newcommand{\espc}{\vspace{0.2cm}}

\begin{document}

\title{Partially hyperbolic dynamics in dimension 3}

\author[P. Carrasco]{Pablo D. Carrasco}
\address[P. Carrasco]{ICMC-USP, Avenida Trabalhador S\~{a}o-carlense 400, S\~{a}o Carlos, SP 13566-590, Brazil}
\email{pdcarrasco@gmail.com}
\thanks{The first author is supported by FAPESP Project\# 2013/16226-8.}
\author[F. Rodriguez-Hertz]{Federico Rodriguez-Hertz}
\address[F. Rodriguez-Hertz]{PSU Math. department, University Park, State College, PA 16802, US}
\email{hertz@math.psu.edu}
\author[J. Rodriguez-Hertz]{Jana Rodriguez-Hertz}
\address[J.Rodriguez-Hertz and R. Ures]{IMERL-FING, Julio Herrera y Reissig 565, Montevideo 11300, Uruguay}
\email{jana@fing.edu.uy}
\author[R. Ures]{Ra\'{u}l Ures}
\email{ures@fing.edu.uy}





\begin{abstract}
Partial hyperbolicity appeared in the sixties as a natural generalization of hyperbolicity. In the last 20 years in this area there has been great activity. Here we survey the state of the art in some topics, focusing especially in partial hyperbolicity in dimension 3. The reason for this is not only that it is the smallest dimension in which non-degenerate partial hyperbolicity can occur, but also that the topology of 3-manifolds influences this dynamics in revealing ways. 
\end{abstract}
\maketitle
\section{Introduction}
Partial hyperbolicity was introduced in the late sixties as a generalization of the classical notion of hyperbolicity. In hyperbolic systems, the tangent bundle splits into two directions that are invariant under the derivative, one, the stable direction, contracted, and the other, the unstable direction, expanded. More precisely, a diffeomorphism of a compact manifold $f:M\to M$ is {\em hyperbolic} if the tangent bundle splits as $TM=E^{s}\oplus E^{u}$, where $Df(x)E^{s}_{x}=E^{s}_{f(x)}$ and $Df(x)E_x^{u}=E^{u}_{f(x)}$, and for each pair of unit vectors $v^{s}\in E^{s}_{x}$ and $v^{u}\in E^{u}_{x}$ we have:

$$\|Df(x)v^{s}\|<1<\|Df(x)v^{u}\|$$ 
The simplest examples of this behavior are the hyperbolic (also known as Anosov) automorphisms on tori.

Partially hyperbolic diffeomorphisms, in turn, allow one extra, center direction, which is neither as expanded as the unstable one nor as contracted as the stable one. Again the simplest examples are certain automorphisms of the torus, for instance: $$A=\left( \begin{array}{ccc} 2 & 1 & 0 \\
                         1 & 1 &  0 \\ 0 & 0 & 1  \end{array}\right).$$
This matrix has three eigenvalues: $\lambda>1$, $\lambda^{-1}$ and $1$. Thus, the associated automorphism of $\T^3$ has three invariant bundles paralell to its 3 eigendirections.  Two of these bundles, the ones associated to $\lambda$ and $\lambda^{-1}$, have a  hyperbolic behavior and the one associated to $1$ corresponds to the center direction. 

Another classical (and more interesting)  example in dimension 3 comes from the action of the diagonal subgroup on quotients of $G=PSL(2,\R)$.  Consider a left invariant riemannian metric on $G$.  \label{diagonal.subgroup}

{\it \bf Statement:}\label{action.subgroup} The right multiplication by $d_1=\left( \begin{array}{cc} e^\frac 12 & 0 \\ 0 & e^{-\frac 12} \end{array}\right)$ induces a partially hyperbolic diffeomorphism of $G$, $\psi:=R_{d_1}:G\to G$, $\psi(g):=R_{d_1}(g)=gd_1$.

Let $$u^-_t=\left( \begin{array}{cc} 1 & t \\ 0 & 1 \end{array}\right), u^+_t=\left( \begin{array}{cc} 1 & 0\\ t& 1 \end{array}\right) ; \,\, d_t=\left( \begin{array}{cc} e^{\frac t2} & 0 \\ 0 & e^{-\frac t2} \end{array}\right) ;  $$ be respectively, the stable horcocycle, the unstable horocycle and the diagonal 1-parameter groups. Consider the foliations generated by these 1-parameter groups $$\mathcal{F}^s(g)=\left\{ u^-_tg  ; \,\, t\in \R\right\}; \,\, \mathcal{F}^u(g)=\left\{ u^+_tg  ; \,\, t\in \R\right\}; \,\, \mathcal{F}^c(g)=\left\{ d_tg  ; \,\, t\in \R\right\}.$$ $\psi$ respectively contracts, expands and is an isometry on the leaves of these foliations. Notice that \begin{eqnarray}\label{commutativity}u^\pm_sd_t=d_tu^\pm_{e^{\pm t}s}d_t.\end{eqnarray} This equation shows that $\psi$ intertwines leaves of the foliation $\mathcal F^s$ and similarly with $\mathcal F^u$, moreover $\psi$ preserves the leaves of $\mathcal F^c$. We also get from equation \ref{commutativity} that $$\psi(gu^\pm_s)=gu^\pm_sd_1=gd_1u^\pm_{e^{\pm 1}s}=\psi(g)u^\pm_{e^{{\pm 1}}s}.$$ Remember we choose a left invariant metric on $G$, so, using its distance function we obtain that \begin{eqnarray}\label{hyp}d(\psi(gu^\pm_s),\psi(g))=d(gd_1u^\pm_{e^{\pm 1}s},gd_1)=d(u^\pm_{e^{\pm 1}s},e)=d(gu^\pm_{e^{\pm 1}s},g)\end{eqnarray}
 and 
\begin{eqnarray}\label{isom}d(\psi(gd_t),\psi(g))=d(gd_td_1,gd_1)=d(gd_1d_t,gd_1)=d(d_t,e)=d(gd_t,g).\end{eqnarray}
From equation \ref{hyp} we obtain that the leaves of $\mathcal F^s$ are exponentially contracted in the future and the leaves of $\mathcal F^u$ are exponentially contracted in the past by $\psi$. Equation \ref{isom} shows that $\psi$ acts isometrically on the leaves of $\mathcal F^c$. Define $E^\sigma_g=T_g\mathcal F^\sigma$, $\sigma=s,u,c$ and observe that we obtain a splitting $$E^s_g\oplus E^u_g\oplus E^c=T_gG$$ where the first direction is exponentially contracted, the second is exponentially expanded and the third one is isometric. 

Let us give a different description of the invariant bundles. Let us describe first what would be the partially hyperbolic splitting for the tangent space to the identity element $e=\left( \begin{array}{cc} 1 & 0\\ 0 & 1 \end{array}\right)$. The tangent space to the  identity element $T_e G$ is naturally identified with the Lie algebra of $G$, $\mathfrak{g}=\mathfrak{sl_2}$. On $\mathfrak{g}$ we have three distinguished elements $$U^-=\left( \begin{array}{cc} 0 & 1 \\ 0 & 0 \end{array}\right);\;\;\; U^+\left( \begin{array}{cc} 0 & 0 \\ 1 & 0 \end{array}\right);\;\;\; D=\left( \begin{array}{cc} \frac12 & 0 \\ 0 & -\frac12 \end{array}\right).$$ Observe that \begin{eqnarray}[D,U^+]=U^+;\;\;\; [D,U^-]=-U^-;\;\;\;[U^+,U^-]=2D.\end{eqnarray} Let $E^u_e=\hbox{span}\{U^+\}$,  $E^s_e=\hbox{span}\{U^-\}$ and $E^c_e=\hbox{span}\{D\}$. Clearly $U^+,U^-,D$ are linearly independent and hence $$E^s_e\oplus E^u_e\oplus E^c_e=\mathfrak g= T_e G.$$ Given $g\in G$ let us define $R_g:G\to G$ by $R_g(h)=hg$. Let $D_eR_g:T_eG\to T_g G$ be the derivative of $R_g$ over the identity and define $E^\sigma_g= D_eR_g(E^{\sigma}_e), \sigma=s,c,u$. Let us show that these bundles are accordingly the unstable, stable and center bundle. To this end we shall use $L_g:G\to G$, $L_g(h)=gh$, the left translation. Since the riemannian metric we choose was a left invariant metric, we get that $L_g$ is an isometry for every $g$. Observe that $(L_g)^{-1}=L_{g^{-1}}$. 

Both $\psi$ and the right group actions remain well-defined if we take a quotient of $PSL(2,\R)$ under the left action of a lattice $\Gamma$ and since the riemannian metric we choose is left invariant, we obtain a riemannian metric on the quotient $\Gamma\backslash PSL(2,\R)$. The contraction and expansion properties are hence preserved on this quotient also. Recall that the 3-manifold $\Gamma\backslash PSL(2,\R)$ can be naturally identified with the unit tangent bundle of a closed surface (compact in case $\Gamma$ is co-compact) and $\psi$ with the time-one map of the geodesic flow of a metric of constant negative curvature on this surface (again the extra direction is given by the direction of the diagonal flow). See \cite{Katok-Hasselblatt}.

Precisely, we say that a diffeomorphism $f$ of a closed manifold $M$ is partially hyperbolic if the tangent bundle splits as $TM=E^{s}\oplus E^{c}\oplus E^{u}$, where $Df(x)E^{u}_x=E^{u}_{f(x)}$, $Df(x)E^{c}_x=E^{c}_{f(x)}$ and $Df(x)E^{s}_x=E^{s}_{f(x)}$, and for each unit vector $v^{s}\in E^{s}_{x}$, $v^{c}\in E^{c}_x$ and $v^{u}\in E^{u}_{x}$ we have:
\begin{eqnarray}
\|Df(x)v^{s}\|<&1&<\|Df(x)v^{u}\|,\quad\text{and}\label{partial.hyperbolicity1}\\
\|Df(x)v^{s}\|<&\|Df(x)v^{c}\|&<\|Df(x)v^{u}\|. \label{partial.hyperbolicity2}
\end{eqnarray}

 The set of partially hyperbolic diffeomorphisms is $\mathcal{C}^1$-open in $\diff^1(M)$. See, for instance, Theorem 2.15 in \cite{HPS}. In other words, a $C^{1}$-perturbation of a partially hyperbolic diffeomorphism is partially hyperbolic. 

\espc

\subsection{More examples}\label{subsection.examples}
The examples given in the Introduction fall into more general classes:
\subsubsection{ Time-one maps of Anosov flows:}\label{example1} Consider an Anosov flow in a 3-manifold $\phi_{t}:M\to M$ such that the tangent bundle of $M$ splits into 3 sub-bundles invariant under $D\phi_{t}$: $TM=E^{s}\oplus X\oplus E^{u}$, where $X$ is the direction tangent to the flow, and such that for each unit vector $v^{s}\in E^{s}$ and $v^{u}\in E^{u}$
$$\|D\phi_{t}v^{s}\|<1<\|D\phi_{t}v^{u}\|.$$
Then the time-one map of the flow $\phi_{t}$ is a partially hyperbolic diffeomorphism (exercise).\par
Examples of this kind of partially hyperbolic diffeomorphism are the ones induced by the diagonal action on $PSL(2, \R)$ mentioned above.  \par
As a different type of example consider the suspension of an Anosov map $A:\mathbb{T}^2\rightarrow\mathbb{T}^2$\ by the
constant roof function one: Let $A$ be a hyperbolic automorphism on $\mathbb{T}^{2}$, and consider in $\mathbb{T}^{2}\times \R$ the following equivalence relation: $(x,t+1)\sim (Ax,t)$. Then $M=\mathbb{T}^{2}\times \R|_{\sim}$ is a smooth manifold, and $f([x,t])=[x,t+1]$ is a partially hyperbolic diffeomorphism. \par

We remark that these examples are truly different: for the second one the distribution $E^u\oplus E^s$ is integrable whereas for the first one it is not.\par
Note that in both examples the corresponding Anosov flow is transitive, but there also exist non transitive Anosov flows \cite{FW80}. There is in fact a huge zoo of Anosov flows, see for instance \cite{barbot98}, \cite{BL}, \cite{fenley94}, \cite{fried83}, \cite{HT}; so classification of time-one maps of these flows is naturally a difficult task.

\subsubsection{ Skew-products:} Another example of partially hyperbolic diffeomorphism is a certain kind of skew-product that is a circle extension over the two torus of the form $f(x, \theta)=(Bx, h(x, \theta))$ where $B$ is a hyperbolic automorphism of the 2-torus and $h(x,.)$ are circle rotations. The resulting ambient manifold is a 3-nilmanifold. In case the product is direct, the ambient manifold is the 3-torus. To this class belongs $A$, the 3-toral automorphism defined at the beginning of this introduction. 
  
\subsubsection{DA-diffeomorphisms:} A DA- partially hyperbolic diffeomorphism is one that is isotopic to an Anosov one. By a result of Franks, DA-diffeomorphisms are semi-conjugate to Anosov diffeomorphisms. 
 The prototypical example in this class is Ma{\~{n}}{\'{e}}'s example \cite{Contributions}. It is obtained by taking
a linear Anosov map in $\mathbb{T}^3$\ with eigenvalues $\lambda^{ss}<1<\lambda^u<\lambda^{uu}$, and making a bifurcation of the origin
into three points along the weak unstable direction (See fig. \ref{DAa}). The resulting map is a dynamically coherent
(robustly) transitive partially hyperbolic diffeomorphism.

\begin{figure}[ht]\label{DAa}
\centering
  \includegraphics[width=10cm]{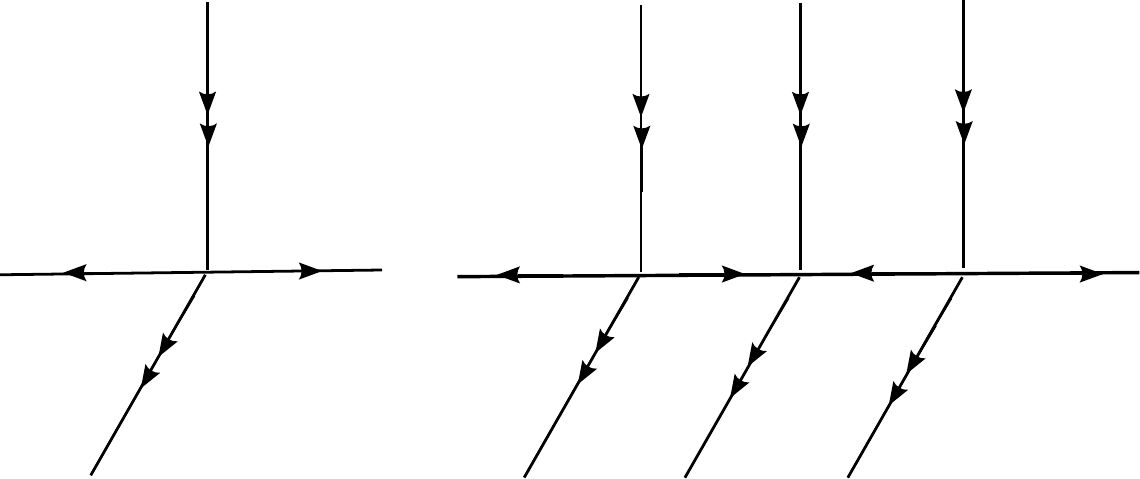}\\
  \caption{Ma{\~{n}}{\'{e}}'s DA-partially hyperbolic diffeomorphism}\label{DA}
\end{figure}
\subsection*{Acknowledgements}{The authors want to thank Amie Wilkinson for her careful review of this manuscript and for many valuable suggestions she made that helped to improve its readability.}


\section{Some open problems}
Partial hyperbolicity appeared as a natural generalization of hyperbolicity. 
One main motivation to study partially hyperbolic systems was ergodicity. A system $f:M\to M$ preserving a measure $m$ is ergodic when, in average, any two events tend to be independient, that is, for any two measurable sets $A$ and $B$ we have
$$\lim_{n\to\infty}\frac1N\sum_{n=0}^{N-1} m(f^{n}(A)\cap B) =m(A)m(B)$$
Ergodicity is therefore an interesting property from the physical point of view. 
Since conservative hyperbolic systems are ergodic \cite{AS67}, it was asked whether adding some hyperbolicty to a system and then perturbing would generate ergodicity in the whole system. Concretely, Pugh and Shub \cite{stableerg} asked if taking any system, taking the product of it by a sufficiently strong hyperbolic system (hence obtaining a partially hyperbolic system) and then making a small perturbation would yield ergodicity in a robust way. They went further to state Conjecture \ref{conjecture.ps} below. 
The area of partial hyperbolicity has become very active since then, and other aspects have attracted interest as well. These are also stated below.

\subsection{Ergodicity} As stated above, one problem in partially hyperbolic dynamics is studying the ergodicity of conservative systems, to be more precise, of $C^{r}$- diffeomorphisms preserving a smooth volume. We shall denote conservative systems by $\diff^{r}_{m}(M)$, where $m$ denotes the probability measure arising from this smooth volume. An equivalent definition of ergodic system $f$ is that any measurable set $A$ satisfying $f(A)=A$ must also satisfy either $m(A)=1$ or $m(A)=0$. In other words, an ergodic system is one not admitting an invariant measurable set with intermediate measure.  \par
An example of a non-ergodic partially hyperbolic diffeomorphism is the toral automorphism $A$ defined at the beginning of the introduction. This automorphism can be seen in the following way:  $A=B\times id$, where 
$B$ is the automorphism on the 2-torus generated by the matrix $\left(
\begin{array}{cc}
2&1 \\1&1
\end{array}
\right)$, and $id$ is the identity on the circle. Indeed, any interval of the circle times the 2-torus is an invariant set with intermediate measure. \par
One can easily perturb this system to obtain an {\em ergodic} partially hyperbolic one. For instance $g=B\times$(irrational rotation). However, this new system $g$ can be easily perturbed to obtain again non-ergodic diffeomorphisms. It is, in fact approximated by diffeomorphisms $g'=B\times$(rational rotation), which are non-ergodic. \par
Pugh and Shub were the first to conjecture that ergodicity is in fact very abundant in the partially hyperbolic world, a conjecture that remains open today (Conjecture \ref{conjecture.ps}). As a matter of fact, this conjecture was made public for the first time in Montevideo, in 1995 \cite{pughshub95}. We thank Keith Burns for recalling this fact, and Mike Shub for confirming it. Pugh and Shub considered not only ergodicity, but a stronger concept, namely:
\begin{definition}\label{def.stable.ergodicity}
A conservative $\mathcal{C}^2$ diffeomorphism $f:M\rightarrow M$ is stably ergodic (in $\diff^1_m(M)$) if there exists $U\subset \diff^1_m(M)$, a  neighborhood of $f$, such that every $g\in U$ of class $\mathcal{C}^2$ is ergodic.
\end{definition}
Until 1994, the only known examples of stably ergodic diffeomorphisms were Anosov diffeomorphisms, that is, hyperbolic ones. In 1994, Grayson, Pugh and Shub found the first non-hyperbolic examples.\par
A year later, Pugh and Shub conjectured the following:

\begin{conjecture}[Pugh-Shub (1995)\cite{pughshub95}] \label{conjecture.ps}\label{pugh.shub.conjecture} Stable ergodicity is $C^r$-dense among volume preserving partially hyperbolic diffeomorphisms, for all $r>1$.
\end{conjecture}

Pugh and Shub suggested a program in order to prove their conjecture. It involves accessibility, defined below. 
\begin{definition}
Two points $x$ and $y$ are in the same {\em accessibility class} if there is a path piece-wise tangent to $E^{s}$ or $E^{u}$ joining them. The  partially hyperbolic diffeomorphism $f$ has the {\em accessibility property} if there is only one accessibility class. The diffeomorphism has the {\em essential accessibility property} if any measurable set that is a union of accessibility classes has either full or null measure. 
\end{definition}
In the picture below, the points $x,y$ and $z$ are in the same accessibility class. 
\begin{center}
\begin{figure}[h]
 \includegraphics[width=.5\textwidth]{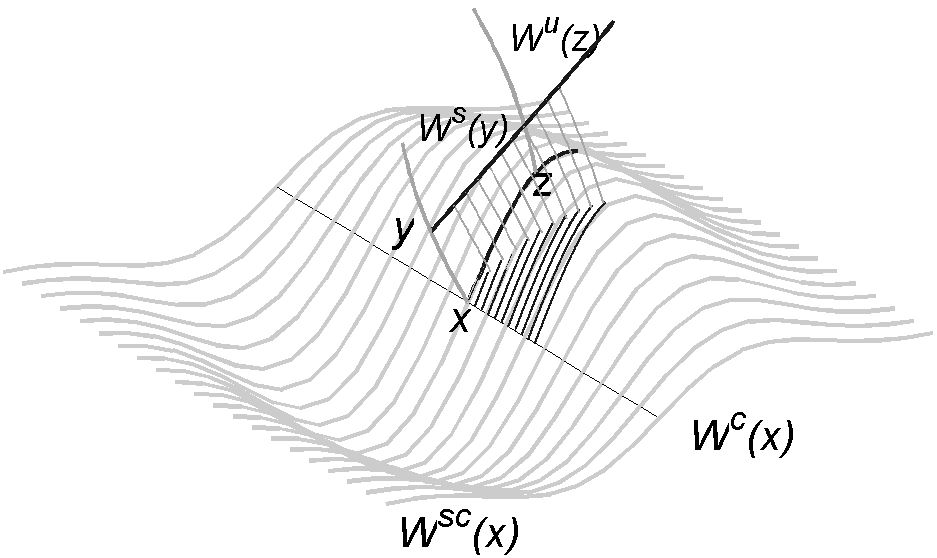}
\end{figure}
\end{center}
\begin{exercise}
Prove that the example of the action of the diagonal subgroup defined in Page \pageref{action.subgroup} has the accessibility property. 
\end{exercise}
Obviously, accessibility implies essential accessibility, but the converse is not true.
\begin{exercise}
Show that the automorphism on $\mathbb{T}^{3}$ defined by $$A=\left(
\begin{array}{rrr}
0&0&1\\ 0&1&-1\\-1&-1&5 
\end{array}
\right)$$
 has the essential accessibility property, but it does not have the accessibility property.
\end{exercise}

\begin{conjecture}[Pugh-Shub] \label{conjecture.acc.implies.erg} If $f$ is a $C^{r}$ conservative diffeomorphism, with $r>1$, (essential) accessibility implies ergodicity.
\end{conjecture}

\begin{conjecture}[Pugh-Shub] \label{conjecture.stable.acc.dense} Stable accessibility is $C^{r}$-dense among volume preserving partially hyperbolic diffeomorphisms, for $r>1$.
\end{conjecture}

In \cite{burnswilkinson}, Burns and Wilkinson proved Conjecture \ref{conjecture.acc.implies.erg} under the additional condition of {\em center bunching} (which means that the hyperbolicity of the $E^{s}\oplus E^{u}$-bundle is stronger than the non-conformality of $E^{c}$). Previously, Dolgopyat-Wilkinson \cite{DW03} had proven that stable accessibility is $C^{1}$-dense. 
Recently, Avila, Crovisier and Wilkinson have shown that stable ergodicity is $C^{1}$-dense among $C^{r}$ partially hyperbolic diffeomorphisms, with $r>1$. 

\begin{theorem}[Avila-Crovisier-Wilkinson \cite{acw}] Stable ergodicity is $C^{1}$-dense among $C^{r}$ volume preserving partially hyperbolic diffeomorphisms, $r>1$.
\end{theorem}

\par
In their work, they build on an approach by \cite{hhtu} where this result was proven for the case of two-dimensional center bundle. In \cite{hhtu} there were introduced the {\em Pesin homoclinic classes}, which were shown to be hyperbolic ergodic components of $m$, these classes were made into one by the use of blenders. The use of this technique for center bundles of higher dimensions is a non-trivial fact, which was overcome by \cite{acw} by using a different kind of blenders, introduced by Moreira and Silva \cite{guguzulu}.\par
Other important advance is  Burns-Dolgopyat-Pesin \cite{BDP02}, who proved some version of \ref{conjecture.acc.implies.erg}: essential accessibility and positive center Lyapunov exponents imply stable ergodicity. \par
In \cite{HHU2008inv}, Conjecture \ref{pugh.shub.conjecture} was proven for one-dimensional center bundle. In particular, it holds: 
\begin{theorem}[Hertz-Hertz-Ures \cite{HHU2008inv}] Stable ergodicity is $C^{\infty}$-dense among volume preserving partially hyperbolic diffeomorphisms on 3-manifolds. 
\end{theorem}

In sum, the vast majority of 3-dimensional partially hyperbolic diffeomorphisms are ergodic and moreover stably ergodic. Hence one could ask the following: can we classify the {\em non-ergodic} partially hyperbolic diffeomorphisms? And also: are there manifolds where {\em all} partially hyperbolic diffeomorphisms are ergodic? Can we classify all the 3-manifolds admiting non-ergodic partially hyperbolic diffeomorphisms?\par

A first evidence in this direction was obtained in \cite{ParHypDim3}:

\begin{theorem}[Hertz-Hertz-Ures] \label{theorem.hhu.nil} If $N$ is a 3-nilmanifold other than $\T^{3}$, then all conservative partially hyperbolic diffeomorphisms are ergodic.
\end{theorem}

The proof of this theorem involves the study of accessibility classes defined above. It follows from \cite{burnswilkinson} and \cite{HHU2008inv} that for $C^{2}$-partially hyperbolic diffeomorphisms in 3-dimensional manifolds, accessibility implies ergodicity. This fact is very interesting, since it allows to convert an ergodic problem into a geometric problem: the study of the set of accessibility classes; in other words, if the system has only one accessibility class, then it is ergodic. In Section \ref{section.ergodicity} a better description of these sets in 3-manifolds is given. \newline \par
While proving Theorem \ref{theorem.hhu.nil}, we got the impression that the only obstruction to ergodicity is the existence of a proper {\em compact} accessibility class:

\begin{conjecture}[{\bf Ergodic conjecture:} Hertz-Hertz-Ures (2008)] \label{conjecture.strong.non.ergodic} If a conservative partially hyperbolic diffeomorphism of a 3-manifold is non-ergodic, then there is a 2-torus tangent to $E^{s}\oplus E^{u}$.
\end{conjecture}

The importance of this kind of hyperbolic sub-dynamics will become apparent later on, see Section \ref{section.anosov.tori}. As we will see, these tori seem to be ``behind'' a lot of interesting behavior in partially hyperbolic dynamics. And moreover, not every orientable manifold can support any such sub-dynamics. In order to describe the 3-manifolds admitting these 2-tori, let us recall the concept of {\em mapping torus}: If $N$ is a closed manifold, and $g:N\to N$ is a  diffeomorphism, we define the mapping torus of $g$ as the manifold obtained by identifying in $N\times [0,1]$ the points $(x,1)$ with $(g(x),0)$.\par
As examples of these we can mention the 3-nilmanifolds, which can be seen as the mapping tori of automorphisms of the form $\left(
\begin{array}{rr}
1&k\\0&1 
\end{array}
\right),$ with $k\in\mathbb{Z}$. The particular case $k=0$ gives the 3-torus. Examples of solvmanifolds are the mapping tori of hyperbolic automorphisms $A:\mathbb{T}^{2}\To\mathbb{T}^{2}$.

\begin{theorem}[\cite{AnosovTori}]\label{thm.anosov.tori} If a partially hyperbolic diffeomorphism $f:M^{3}\to M^{3}$ has a 2-dimensional embedded torus $T$ which is tangent either to $E^{s}\oplus E^{u}$, $E^{c}\oplus E^{u}$ or $E^{c}\oplus E^{s}$, then the ambient manifold $M^{3}$ can only be one of the following possibilities:
\begin{enumerate}
\item the 3-torus $\T^{3}$
\item the mapping torus of $-id:\T^{2}\to\T^{2}$
\item the mapping tori of hyperbolic automorphisms on 2-tori 
\end{enumerate}
\end{theorem}

\begin{figure}[ht]
\centering
\subfigure{
\includegraphics[scale=0.15]{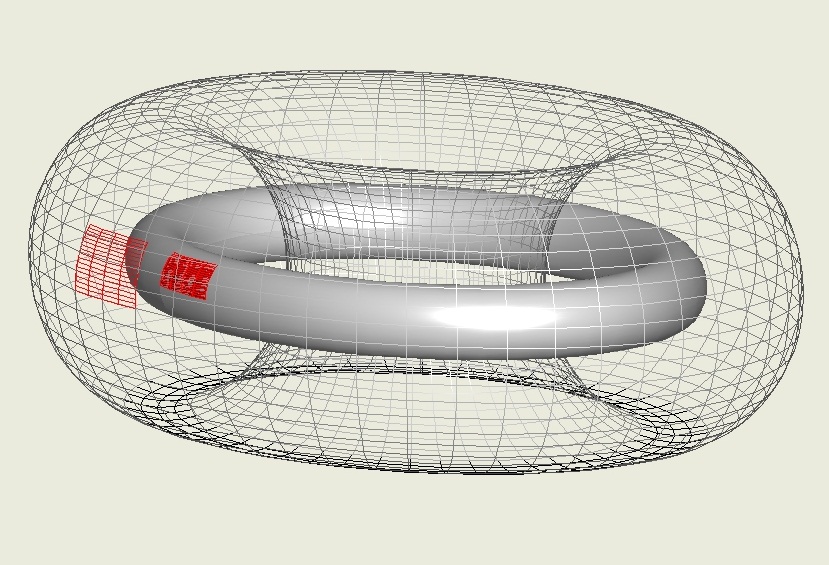}
}
\subfigure{
\includegraphics[scale=0.15]{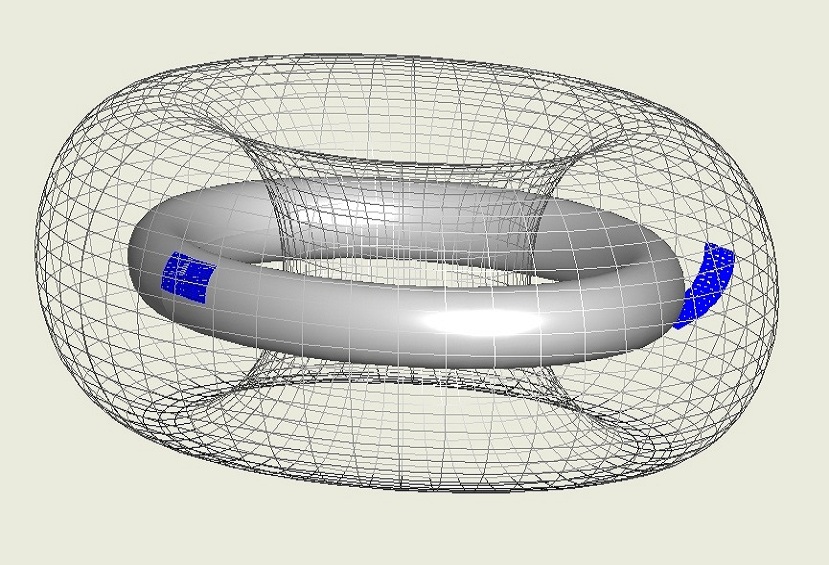}
}
\subfigure{
\includegraphics[scale=0.15]{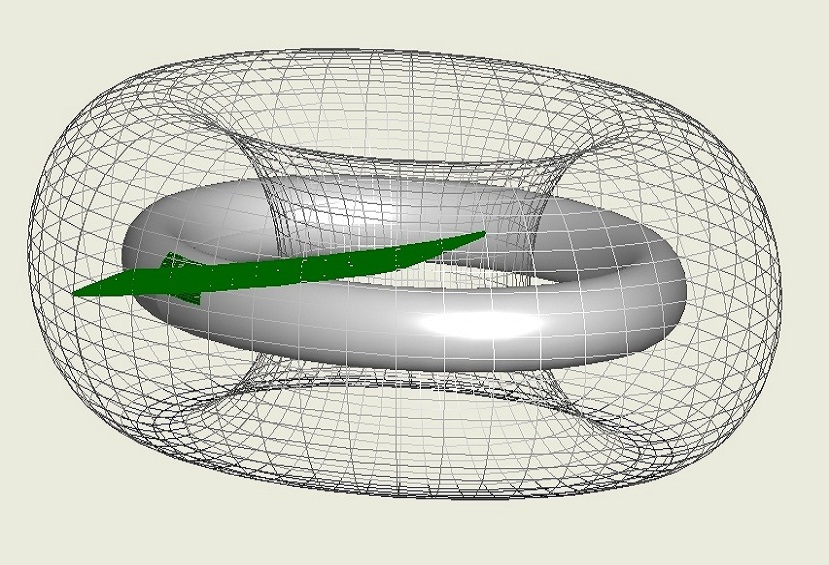}
}

\caption{(1) The $3$-torus (2) The mapping torus of $-Id$ (3) The mapping torus of an hyperbolic automorphism.}
\end{figure}

In the three cases it is possible to find partially hyperbolic dynamics with embedded tori tangent to $E^{s}\oplus E^{u}$ or $E^{c}\oplus E^{u}$ ($E^{c}\oplus E^{s}$ is analogous). We refer the reader to Section \ref{section.dc} to see how a torus tangent to $E^{c}\oplus E^{u}$ can be built in any of these examples. Finding a partially hyperbolic diffeomorphism with a torus tangent to $E^{s}\oplus E^{u}$ in cases (1) and (3) is trivial, in case (2), it is enough to consider a hyperbolic automorphism on $\mathbb{T}^{2}$, $B$, and take the diffeomorphism $f([x,t])=[Bx,t]$. This is the desired partially hyperbolic diffeomorphism.

If Conjecture \ref{conjecture.strong.non.ergodic} is true, there would be very few 3-manifolds supporting non-ergodic partially hyperbolic dynamics. We state this explicitly as a weaker conjecture:

\begin{conjecture}[{\bf Weak ergodic conjecture:} Hertz-Hertz-Ures (2008)] \label{conjecture.non.ergodic}The only orientable 3-manifolds that admit a non-ergodic conservative partially hyperbolic diffeomorphism are:
\begin{enumerate}
\item the 3-torus $\T^{3}$
\item the mapping torus of $-id:\T^{2}\to\T^{2}$
\item the mapping tori of hyperbolic automorphisms on 2-tori 
\end{enumerate} 
\end{conjecture}
The following is also open:
\begin{problem}
 Let $f$ be a conservative non-ergodic partially hyperbolic diffeomorphism in any of the manifolds (1), (2), (3) stated in Theorem \ref{thm.anosov.tori}. Prove that there exists a torus tangent to $E^{s}\oplus E^{u}$. 
\end{problem}
In Section \ref{section.ergodicity} we collect the advances in this conjecture to the best of our knowledge, and a description of the set of accessibility classes. 

\subsection{Dynamical coherence} In partially hyperbolic dynamics, the strong bundles $E^{s}$ and $E^{u}$ are always integrable; that is, there are invariant foliations $W^{s}$ and $W^{u}$, the {\em stable} and {\em unstable} foliations, that are tangent to each of the strong bundles (see, for instance \cite{brinpesin1974,HPS}). However, the center bundle, $E^{c}$, can be integrable or not. In fact, as Wilkinson noticed in \cite{wilkinson1998}, the Anosov example of A. Borel, mentioned by Smale in \cite{smale1967}, when viewed as a partially hyperbolic diffeomorphism, has a non-integrable center bundle. The Borel-Smale example is explained in detail in \cite{burnswilkinson2008}. For the sake of completeness, let us briefly mention what it is about. \par
\subsubsection{A non-dynamically coherent example} The aforementioned example is, in fact, a non-toral Anosov automorphism on a six-dimensional nilmanifold. Let $G_{1}$ and $G_{2}$ be copies of the three dimensional simply connected nonabelian nilpotent Lie group. And consider bases $X_{i}$, $Y_{i}$,  $Z_{i}$ of the corresponding Lie algebras ${\mathcal G}_{i}$, $i=1,2$, with the bracket condition
\begin{equation}\label{equation.bracket} [X_{1},Y_{1}]=Z_{1}\qquad[X_{2},Y_{2}]=Z_{2}\end{equation}
Consider now a hyperbolic automorphism $A\in SL(2,\Z)$, and let  $\lambda>1$ be one of its eigenvalues. Consider $f$ such that the derivative acts over the Lie algebra as follows:

\begin{eqnarray*}
 X_{1}\mapsto \lambda X_{1}&& X_{2}\mapsto \lambda^{-1} X_{2}\\
 Y_{1}\mapsto \lambda^{2}Y_{1}&& Y_{2}\mapsto \lambda^{-2} Y_{2}\\
  Z_{1}\mapsto \lambda^{3}Z_{1}&& Z_{2}\mapsto \lambda^{-3} Z_{2}
\end{eqnarray*}

The next step (which we will not explain) is to find a lattice $\Gamma$ of $G=G_{1}\times G_{2}$, such that it is invariant, so that the whole construction yields a diffeomorphism over the six-dimensional nilmanifold $G/\Gamma$ (the coset space). 

This example is, as a matter of fact, an Anosov diffeomorphism. But we may also look at it as a partially hyperbolic diffeomorphism, such that $E^{s}$ is the space generated by $Z_{2}$, $E^{u}$, the space generated by $Z_{1}$, and the center bundle, $E^{c}$, is the space generated by $X_{1}, X_{2}, Y_{1}$ and $Y_{2}$. \par
Note that $X_{1},X_{2},Y_{1}$ and $Y_{2}$ do not satisfy the Frobenius condition, due to (\ref{equation.bracket}), therefore $E^{c}$ is not integrable, that is, there is no invariant foliation tangent to $E^{c}$. \newline \par
After examining this example, we may ask: is the lack of Frobenius condition the only reason for non-integrability of the center bundle? What about the case of one-dimensional $E^{c}$, where the Frobenius condition is always trivially satisfied?  Notice that $E^{c}$ is a priori only H\"older, so Picard's Theorem does not necessarily hold. The question of whether a partially hyperbolic diffeomorphism existed with a one-dimensional non-integrable center bundle remained open since the 70's. \par
Let us define a stronger concept of integrability of the center bundle.

\begin{definition}\label{definition.dc}
A partially hyperbolic diffeomorphism is {\em dynamically coherent} if the following two conditions are satisfied:
\begin{enumerate}
 \item There is an invariant foliation tangent to the distribution $E^{c}\oplus E^{u}$
 \item There is an invariant foliation tangent to the distribution $E^{s}\oplus E^{c}$.
\end{enumerate}
\end{definition}

In 2009, Hertz-Hertz-Ures \cite{HHUnonDCexample}  provided the first example of a partially hyperbolic diffeomorphism with a non-integrable one-dimensional center-bundle. It is a non-dynamically coherent example in a 3-torus.  Let us mention that this example strongly contrasts with a result by Brin, Burago and Ivanov \cite{brinburagoivanov2009}, which proves that all {\em absolutely} partially hyperbolic diffeomorphisms of the 3-torus are dynamically coherent. Absolute partial hyperbolicity is more restrictive than partial hyperbolicity and requires the bound in (\ref{partial.hyperbolicity2}) to be uniform, namely: that there exist $\lambda<1<\mu$ such that for all $x\in M$ and unit vectors $v^{\sigma}\in E^{\sigma}_{x}$, $\sigma=s,c,u$ we have 
\begin{eqnarray}\label{eq.abs.ph}
\|Df(x)v^{s}\|\leq \lambda \leq \|Df(x)v^{c}\|\leq \mu\leq \|Df(x)v^{u}\|.
\end{eqnarray}

In \cite{HHUcenterunstable} it was shown that if there is an invariant foliation tangent to the distribution $E^{c}\oplus E^{u}$, then this foliation cannot contain a compact leaf, that is, a 2-torus tangent to $E^{c}\oplus E^{u}$ (see Theorem \ref{teo.toro.cu}). Note that, in principle, a $cu$-torus could coexist with a $cu$-foliation, but cannot be {\em part} of it. \par In this example mentioned below, in fact, the distribution $E^{c}\oplus E^{u}$ is uniquely integrable at every point of $\T^{3}$ minus a 2-torus, which is tangent to $E^{c}\oplus E^{u}$. This shows that no invariant foliation exists which is tangent to $E^{c}\oplus E^{u}$, whence the example is non-dynamically coherent. We refer the reader to Figure \ref{nondinco} and the caption below.

We prove the following:

\begin{theorem}\cite{HHUnonDCexample} \label{teo.non.dc} There exists a partially hyperbolic diffeomorphism $f:\T^{3}\to\T^{3}$ such that 
\begin{enumerate}
\item $E^{c}$ is not integrable
\item There is no invariant foliation tangent to the distribution $E^{c}\oplus E^{u}$
\item There is an invariant 2-dimensional torus $T$  tangent to the distribution $E^{c}\oplus E^{u}$
\end{enumerate}
 Moreover, there is a $C^{1}$-open neighborhood ${\mathcal U}$ of $f$ such that all $g$ in ${\mathcal U}$ satisfy conditions (1) and (2).
\end{theorem}

Again, a 2-dimensional periodic torus appears. How relevant is this object? We conjecture that these tori might be the only obstruction to dynamical coherence. 

\begin{conjecture}[{\bf Dynamical coherence conjecture}: Hertz-Hertz-Ures (2009)] \label{non.dc.conjecture.strong}If a partially hyperbolic diffeomorphism of a 3-manifold is not dynamically coherent, then there is a 2-torus tangent to either $E^{c}\oplus E^{u}$ or $E^{s}\oplus E^{c}$.
\end{conjecture}
Again, if this conjecture were true, the only manifolds supporting non-dynamically coherent dynamics would be the ones listed in Theorem \ref{thm.anosov.tori}. We state this explicitly as a weaker conjecture:

\begin{conjecture}[{\bf Weak dynamical coherence conjecture}: Hertz-Hertz-Ures (2009)] \label{non.dc.conjecture.weak} The only orientable 3-manifolds supporting non-dynamically coherent dynamics are:
\begin{enumerate}
 \item the 3-torus $\T^{3}$
\item the mapping torus of $-id:\T^{2}\to\T^{2}$
\item the mapping tori of hyperbolic automorphisms on 2-tori 
\end{enumerate} 
\end{conjecture}

Conjecture \ref{non.dc.conjecture.strong} was proven true in the 3-torus by Potrie in his PhD. Thesis \cite{potrie}. This result was extended to manifolds with solvable fundamental group by Hammerlindl and Potrie in \cite{hammerlindlpotrie}:
\begin{theorem}[Hammerlindl-Potrie  \cite{hammerlindlpotrie}] Let $f$ be a partially hyperbolic diffeomorphism of a 3-manifold with solvable fundamental group. If $f$ is not dynamically coherent, then there exists a 2-torus tangent to 
 either $E^{c}\oplus E^{u}$ or $E^{s}\oplus E^{c}$.
\end{theorem}
Hammerlindl and Potrie hence have proven that if the Weak conjecture \ref{non.dc.conjecture.weak} is true, then the Strong conjecture \ref{non.dc.conjecture.strong} also holds. After Hammerlindl-Potrie, in order to prove the Conjecture  \ref{non.dc.conjecture.strong}, it would be enough to show that all partially hyperbolic diffeomorphisms in 3-manifolds with non-solvable fundamental group are dynamically coherent.\par
Hammerlindl-Potrie's result is the the strongest and most complete one in this direction up to this date. \par
In Section \ref{section.dc} we sketch a proof of this fact for the simplest case. See also Hammerlindl-Potrie's survey \cite{hammerlindl.potrie.survey} for a more complete explanation on this and classification topics. 
\subsection{Classification}\label{subsection.classification}
The third main topic in the study of partially hyperbolic dynamics in 3-manifolds is their classification. 
As early as 2001, Enrique Pujals proposed the following conjecture:

\begin{conjecture}[{\em Classification conjecture:} Pujals (2001), see  \cite{Tranph}] \label{conjecture.pujals}
 If a partially hyperbolic diffeomorphism of a 3-manifold is {\em transitive}, then is is (finitely covered by) one of the following:
\begin{enumerate}
 \item a perturbation of the time-one map of an Anosov flow
 \item a skew product
 \item a DA-diffeomorphism
\end{enumerate}
\end{conjecture}

In 2009, Hammerlindl showed in his PhD Thesis \cite{hammerlindl} that every absolutely partially hyperbolic diffeomorphism of $\T^{3}$ is {\em leafwise conjugate} to its linearization. Let us define this concept:

\begin{definition}[Leaf conjugacy] \label{def.leaf.conjugacy}Two dynamically coherent partially hyperbolic diffeomorphisms $f,g:M\to M$ are {\em leafwise conjugate} if there exists a homeomorphism $h:M\to M$ carrying center leaves of $f$ to center leaves of $g$, that is, $h(W^{c}_{f}(x))=W^{c}_{g}(h(x))$, and such that 
 $$h(f(W_{f}^{c}(x)))=g(h(W_{f}^{c}(x)))$$
\end{definition}

Note that under the hypothesis of Hammerlindl's Thesis (absolute partial hyperbolicity in $\T^{3}$), there is always dynamical coherence, due to a result by Brin, Burago and Ivanov \cite{brinburagoivanov2009}. \par
Hammerlindl's Thesis suggested us that perhaps dynamical coherence was a more suitable hypothesis for a classification of 3-dimensional partially hyperbolic dynamics. This, together with our example \cite{HHUnonDCexample}, led us to propose the following:

\begin{conjecture}[{\bf Classification conjecture:} Hertz-Hertz-Ures (2009)] \label{hhu.conj.classification}Let $f$ be a partially hyperbolic diffeomorphism of a 3-manifold. \par
If $f$ is dynamically coherent, then it is (finitely covered by) one of the following:
\begin{enumerate}
\item a perturbation of a time-one map of an Anosov flow, in which case it is leafwise conjugate to an Anosov flow
\item a skew-product, in which case it is leafwise conjugate to a skew-product with linear base, or
\item a DA, in which case it is leafwise conjugate to an Anosov diffeomorphism of $\T^{3}$.
\end{enumerate}
 If $f$ is not dynamically coherent, then there are a finite number of 2-tori tangent either to $E^{c}\oplus E^{u}$ or to $E^{s}\oplus E^{c}$, and the rest of the dynamics is trivial, as in the non dynamically coherent example \cite{HHUnonDCexample} (see also Section \ref{section.dc})
\end{conjecture}

Both conjectures are true in certain manifolds, as it was proven by Hammerlindl and Potrie \cite{hammerlindlpotrie}:\newline\par
\begin{theorem}[Hammerlindl-Potrie \cite{hammerlindlpotrie}] Conjecture \ref{conjecture.pujals} and \ref{hhu.conj.classification} are true for partially hyperbolic diffeomorphisms in 3-manifolds with solvable fundamental group.
\end{theorem}
The classification conjectures motivated a lot of work; however, very recently, Bonatti, Gogolev, Parwani and Potrie  found both a dynamically coherent example and a transitive example that are not leaf-wise conjugate to any of the above models, proving both classification conjectures wrong \cite{BPP, BGP} (see Section \ref{classification}).
\begin{question} Is it possible to classify partially hyperbolic dynamics in 3-manifolds, modulo leaf conjugacies?
\end{question}

In the next sections we shall develop more deeply the concepts mentioned above. We tried to make each section as self contained as possible, what may imply that some definitions be repeated.

\section{Ergodicity}\label{section.ergodicity}

To establish the ergodicity of partially hyperbolic maps, the most general method available is the so called \emph{Hopf method}. To explain it, we first recall the following.

\begin{theorem}[Stable Manifold Theorem]
Let $M$ be a 3-manifold and let $f\in \diff^r(M)$\ be partially hyperbolic. Then there exist continuous foliations $\mathcal{W}^s=\{W^s(x)\}_{x\in M}$ and $\mathcal{W}^u=\{W^u(x)\}_{x \in M}$\ tangent to
$E^s$ and $E^u$, respectively, called the stable and the unstable foliations. Their leaves are $C^r$-immersed lines.
\end{theorem}

See \cite{HPS}, Theorem $4.1$. We point out that, while their leaves are as smooth as $f$, the foliations $\mathcal{W}^s,\mathcal{W}^u$\ are seldom differentiable (\cite{AnosovThesis}, pag. 201). Their transverse 
regularity is only H\"older (\cite{HolFol}) in general. 

Using the Birkhoff ergodic theorem, it is not hard to see that a conservative diffeomorphism $f$ is ergodic if and only for every {\em continuous} observable $\varphi:M\to \Real$, the Birkhoff average
\begin{equation}\label{equation.birkhoff.global}
\tilde\varphi(x)=\lim_{|n|\to\infty}\frac{1}{n}\sum^{n-1}_{k=0}\varphi\circ f^k(x)
 \end{equation}
 is almost everywhere constant.
But $\tilde\varphi(x)$ coincides almost everywhere with
\begin{equation}\label{equation.birkhoff.+}
\varphi^{+}(x)=\lim_{n\to\infty}\frac{1}{n}\sum_{k=0}^{n-1}\varphi\circ f^{k}(x)
\end{equation}
which is constant on stable manifolds (using uniform continuity of $\varphi$). Analogously, $\tilde\varphi(x)$ coincides almost everywhere with 
\begin{equation}\label{equation.birkhoff.-}
\varphi^{-}(x)=\lim_{n\to\infty}\frac{1}{n}\sum_{k=0}^{n-1}\varphi\circ f^{-k}(x)
\end{equation}
which is constant on unstable manifolds.

Suppose $f$ is not ergodic. Then there would be a continuous observable $\varphi$ for which $\tilde \varphi$ is not almost everywhere constant, and thus neither are $\varphi^{+}$ and $\varphi^{-}$. Hence there exist two positive measure invariant sets $A$ and $B$ and $\alpha \in \Real$ such that $\varphi^{+}(x)\geq\alpha$ for all $x\in A$, and $\varphi^{-}(x)<\alpha$ for all $x\in B$. Note that $A$ is saturated by stable manifolds while $B$ is saturated by unstable manifolds. 

Assume $f$ is an Anosov diffeomorphism. Let $x$ be a point of $A$ such that almost all points $w\in\Ws{x}$  satisfy $\varphi^{-}(w)=\varphi^{+}(w)=\varphi^{+}(x)$. Such an $x$ exists because the stable foliation is {\em absolutely continuous} \cite{AnosovThesis} and $\varphi^+(w)=\varphi^-(w)$ almost everywhere. Consider $y$ a point belonging to the support  of $B$ and $V$ a product neighborhood containing $y$. On the one hand, observe that $\Ws{ x}$ intersects $V$ and let $W_V$ be a connected component of $ \Ws{x}\cap V$.    
On the other hand, $m(V\cap B)>0$ then,  the absolute continuity of the stable foliation implies that there is a local stable manifold $T\subset V$ such that $m_T(T\cap B)>0$ where $m_T$ is the Lebesgue measure of $T$. If we call $h^u$ the unstable holonomy in $V$ sending $T$ to  $W_V$ thus, $m_{W_V}(h^u(T\cap B))>0$. Since $B$ is $u$-saturated we have that $h^u(T\cap B)\subset B$. We obtain a contradiction with the fact that almost every point in $W_V$ satisfies $\varphi^{-}(w)=\varphi^{+}(w)$. This is basically the Hopf argument. 

In brief, there are two fundamental ingredients in the Hopf argument for an Anosov map:
\begin{enumerate}
  \item there is a way of joining any pairs of points through a curve that is piecewise either a stable or an unstable leaf
  \item the stable and unstable foliations are absolutely continuous, and completely transversal.
\end{enumerate}

\subsection{Accessibility, a property that implies ergodicity}

We would like to apply the previous method to a general partially hyperbolic system, that is, when there is some non-trivial center direction. To begin with, observe that in general it is not true that any two pair of points can be joined by a concatenation of stable and unstable leaves. For example, if we consider the partially hyperbolic diffeomorphism $\left(
\begin{array}{cc}
2&1\\1&1 
\end{array}
\right)\times id$ in $\T^{2}\times \mathbb{S}^{1}$, then any path consisting of a concatenation of stable and unstable leaves would be contained in a single 2-torus. 
We fix $f: M\rightarrow M$ $\mathcal{C}^2$ conservative partially hyperbolic, and call $c:[0,1]\rightarrow M$ an \emph{su-path} if it is piecewise $\mathcal{C}^1$ and for every $t$ where defined, $c'(t)\in E^s\cup E^u$.

\begin{definition}  
For a point $x \in M$, its accessibility class is the set 
$$
AC(x):=\{y:\exists c:[0,1]\rightarrow M\text{ su-path such that }c(0)=x,c(1)=y \}.
$$
The map $f$ is accessible if the partition by accessibility classes is trivial, and essentially accessible if the partition by accessibility classes is ergodic (i.e. any Borel set saturated by accessibility classes has either volume $0$ or $1$).
\end{definition}

In the example above, each invariant 2-torus is an accessibility class. When there is only one accessibility class, we will say that $f$ has the {\em accessibility property}. From now on, let us suppose this is our case. As for (2), absolute continuity of the strong foliations is also satisfied (\cite{ErgAnAc}), but complete transversality is not (due to the presence of the center direction)

This problem can be overcome if the holonomies are {\em regular} enough. For instance, Sacksteder used accessibility and Lipschitzness of the stable and unstable holonomies to prove ergodicity of linear partially hyperbolic automorphisms of nil-manifolds \cite{Sac68}. More generally, Brin and Pesin proved that accessibility and Lipschitzness of the stable and unstable foliations imply ergodicity (in fact, Kolmogorov), in the following way \cite[Theorem 5.2,p.204]{brinpesin1974}, see also \cite{GPS}: if $A$ and $B$ are defined as before, consider a density point $x$ in $A$, and a density point $y$ in $B$. Take an $su$-path joining $x$ and $y$. Call $h$ a  global holonomy map from $x$ to $y$, that is, $h$ is a local homeomorphism that takes points in a neighborhood $U$ of $x$, slides them first along a stable segment, then along an unstable, then along a stable again, etc. until reaching a neighborhood $V$ of $y$, all the $su$-paths are near the original $su$-path joining $x$ and $y$. Since $A$ is essentially $su$-saturated, we have that $h(A\cap U)=A\cap V$ modulo a zero set. Since $h$ can be chosen to be Lipschitz, there exists a constant $C>1$ such that, for each measurable set $E\subset U$, and for each sufficiently small $r> 0$, we have
\begin{eqnarray}
  \frac{1}{C}m(E)&<m(h(E))<& Cm(E)\\
  B_{\frac{r}{C}}(y)&\subset h(B_r(x))\subset & B_{Cr}(y).
\end{eqnarray}
This implies that
$$\frac{m(B_{Cr}(y)\cap A)}{m(B_{Cr}(y)\setminus A)}\geq \frac{m(h(B_r(x)\cap A))}{m(h(B_{C^2r}(x)\setminus A))}\geq \frac{1}{C^2.C'}\frac{m(B_r(x)\cap A)}{m(B_r(x)\setminus A)}\to \infty$$
since $m(B_{C^2r}(x)\setminus A)\leq C' m(B_r(x)\setminus A)$ for some positive constant $C'$. From this we get that $y$ is also a density point of $A$. This is absurd, since $y$ was a density point of $B$, complementary to $A$ modulo a zero set. \par
This is essentially how the Hopf argument would work in the partially hyperbolic setting. However, Lipschitzness of the holonomy maps is a very strong hypothesis, not satisfied for most of the partially hyperbolic diffeomorphisms.

The idea of Grayson, Pugh and Shub \cite{GPS}, later improved by \cite{wilkinson1998}, \cite{StableJulienne}, \cite{HHU2008inv}, \cite{burnswilkinson} is to show that the stable and unstable holonomies do preserve density points according to another base different from intervals called \emph{Juliennes}. These sets are dynamically defined, and constitute Vitali bases. 

Burns and Wilkinson made an improvement of this ergodicity argument in \cite{bwpreprint}. We will roughly sketch it. Consider for a point $x$ a small center segment, and saturate by local unstable leaves; to gain better control in the size of these unstable segments we pre-iterate $n$ times the local unstable manifold of $f^n(x)$ (of a convenient size). The resulting set is then saturated by locally stable manifolds.  This small prism is called {\em $s$-julienne}, and denoted by $J^{suc}_n(x)$. The subscript $n$ essentially tells the size of the Julienne, and in particular everything is chosen so that $m(J^{suc}_n(x))\xrightarrow[n\mapsto\infty]{} 0$. An $s$-julienne density point of a set $E$ is a point $x$ such that:
\begin{equation}
  \lim_{n\to\infty}\frac{m(J^{suc}_n(x)\cap E)}{m(J^{suc}_n(x))}=1
\end{equation}
The scheme is to consider the sets $A$ and $B$ we considered above, and prove:
\begin{enumerate}
  \item the $s$-julienne density points of $A$ (and of any essentially $u$-saturated set) coincide with the Lebesgue density points of $A$.
  \item the $s$-julienne density points of $A$ (and of any essentially $s$-saturated set) are preserved by stable holonomies.
\end{enumerate}
An analogous statement is proved for $A$ with respect to $u$-julienne density points, which are defined with respect to the local basis obtained by locally saturating a small center segment first in a dynamic way by stable leaves, and then by unstable leaves. As the stable and unstable holonomies preserve the Lebesgue density points of $A$ we have that if the diffeomorphism has the accessibility property then $A$ is equal to $M$ modulo a zero set. This proves the system is ergodic:

\begin{theorem}[\cite{burnswilkinson,HHU2008inv}]  \label{acc.imp.erg}\label{teo.acc.implies.ergodicity} If $f\in\diff^{2}_m(M^3)$ is partially hyperbolic and satisfies the accessibility property, then it is ergodic.
\end{theorem}

In fact, Burns and Wilkinson prove a much more general result: 
\begin{theorem}[\cite{burnswilkinson}] \label{thm.bw} If $f$ is a partially hyperbolic diffeomorphism (with any center dimension) satisfying the accessibility property and the center bunching property, then $f$ is ergodic
\end{theorem}
A diffeomorphism is said to satisfy the {\em center bunching} property if 

\begin{equation}\label{center.bunching}
\|Df(x)|E^{s}\|<\frac{m(Df(x)|E^{c})}{\|Df(x)|E^{c}\|}\leq \frac{\|Df(x)|E^{c}\|}{m(Df(x)|E^{c})} <m(Df(x)|E^{u})
\end{equation}
where $m(T)=\|T^{-1}\|^{-1}$. \par

 Hence Theorem \ref{thm.bw} implies  Theorem \ref{acc.imp.erg}. Also, when the center bundle is one-dimensional, it is always locally integrable to center curves, so the fake foliations are not necessary to build local juliennes. Hertz-Hertz-Ures in \cite{HHU2008inv} show an alternative way to proving this result in this particular case.

There are two innovations in \cite{burnswilkinson}. One is the argument outlined above Theorem \ref{acc.imp.erg}, which uses a much weaker center bunching than in \cite{StableJulienne} paper. This innovation was already published in \cite{bwpreprint}, and was essential in the proof of Theorem \ref{acc.imp.erg} in \cite{HHU2008inv}. In \cite{bwpreprint}, though, dynamical coherence was still a hypothesis, and the job in Theorem \ref{acc.imp.erg} of \cite{HHU2008inv} consisted in removing this for the center dimension one case. The second innovation in \cite{burnswilkinson} was precisely to remove the dynamical coherence hypothesis for any center dimension. This was accomplished by means of ``fake foliations'', which is a very delicate technical tool.

\subsection{Properties of Accessibility Classes}
We want to precisely describe non-ergodic partially hyperbolic diffeomorphisms, and it is possible that this only occurs when there is a compact accessibility class (see Conjecture \ref{conjecture.strong.non.ergodic}); that is, when there is a torus tangent to $E^{s}\oplus E^{u}$ (in fact we conjecture that there must be at least two such tori).\par
Since accessibility implies ergodicity, in order to describe non-ergodic partially hyperbolic diffeomorphisms, it seems reasonable to look at the non-accessible ones. And, even more precisely, we will study the structure of the set of non-open accessibility classes.\par  

\begin{theorem}\cite{HHU2008inv}\label{teo.ac.open.or.manifold} For each $x\in M^3$, its accessibility class $AC(x)$  is either an open set or an immersed surface. Moreover, $\Gamma(f)$, the set of non-open accessibility classes of $f$ is a compact codimension-one laminated set whose laminae are the accessibility classes.
\end{theorem}

\begin{remark}This theorem still holds for partially hyperbolic diffeomorphisms with center dimension one. 
\end{remark}

Let us begin by a local description of open accessibility classes.
\begin{proposition}\label{proposition.open.ac}
  For any point $x\in M$, the following statements are equivalent:
  \begin{enumerate}
    \item $AC(x)$ is open.
    \item $AC(x)$ has non-empty interior.
    \item $AC(x)\cap W^c_{loc}(x)$ has non-empty interior for any choice of $W^c_{loc}(x)$.
  \end{enumerate}
\end{proposition}
\begin{proof}
  (2) $\tto$ (1) Let $y$ be in the interior of $AC(x)$, and consider any point $z$ in $AC(x)$. Then there is an $su$-path from $y$ to $z$ with points $y=x_0,x_1,\dots,x_N=z$ such that $x_n$ and $x_{n+1}$ are either in the same $s$-leaf or in the same $u$-leaf. Let $U$ be a neighborhood of $y$ contained in $AC(x)$, and suppose that, for instance $y=x_0$ and $x_1$ belong to the same $s$-leaf. Then $U_1=W^s(U)$ is an open set contained in $AC(x)$, that contains $x_1$, so $x_1$ is in the interior of $AC(x)$. Indeed, $W^s$ is a $C^0$-foliation, so the $s$-saturation of an open set is open.\par
  Now, $x_1$ and $x_2$ belong to the same $u$-leaf. If we consider $U_2=W^u(U_1)$, then $U_2$ is an open set contained in $AC(x)$ and containing $x_2$ in its interior. Defining inductively $U_n$ as $W^s(U_{n-1})$ or $W^u(U_{n-1})$ according to whether $x_n$ belongs to the $s$- or the $u$-leaf of $x_{n-1}$, we obtain that all $x_n$ belong to the interior of $AC(x)$. In particular, $z$. This proves that $AC(x)$ is open.\par
  \begin{figure}[h]
  \includegraphics[width=8cm, bb=0 0 685 379]{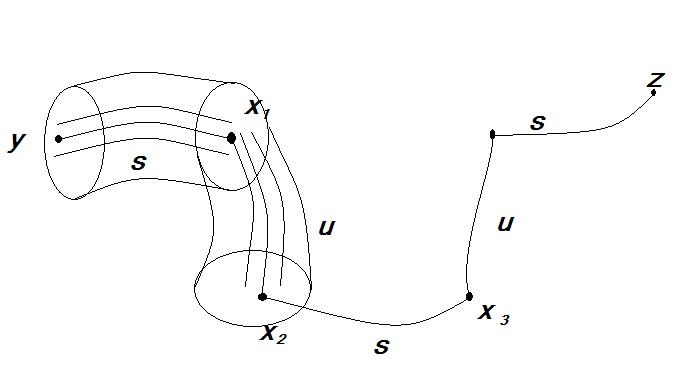}
  \caption{An $su$-path from $y$ to $z$}
  \end{figure}
  (1) $\tto$ (3) Follows directly from the definition of relative topology.\par
  (3) $\tto$ (2) Let $V$ be an open set in $AC(x)\cap W^c_{loc}(x)$, relative to the topology of $W^c_{loc}(x)$. Then $W^s(V)$ is contained in $AC(x)$, and contains a disc $D^{sc}$ of dimension $s+c$ transverse to $E^u$. This implies that $W^u(D^{sc})$ is contained in $AC(x)$ and contains an open set. Therefore, $AC(x)$ has non-empty interior.
\end{proof}

\espc

Let $O(f)$ be the set of open accessibility classes, which is, obviously, an open set. Then its complement, $\Gamma(f)$ is a compact set. Let us see that is laminated by the accessibility classes of its points. \par
For any point $x\in M$, consider a local center leaf $W^c_{loc}(x)$. Locally saturate it by stable leaves, that is, take the local stable manifolds of all points $y\in W^c_{loc}(x)$, to obtain a small $(s+c)$-disc $W^{sc}_{loc}(x)$. Now, locally saturate $W^{sc}_{loc}(x)$ by unstable leaves to obtain a small neighborhood $W^{usc}_{loc}(x)$. See Figure \ref{figure.open.ac}.
\begin{figure}[h]
  \includegraphics[width=6cm, bb=0 0 937 558]{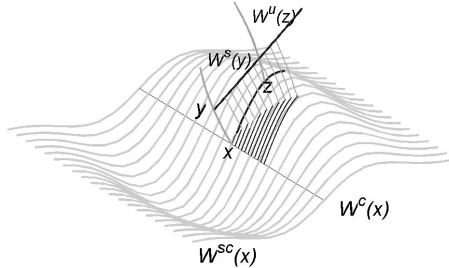}
  \caption{\label{figure.open.ac} An open accessibility class}
\end{figure}
On $W^{usc}_{loc}(x)$, consider the map
\begin{equation}
  p_{us}:W^{usc}_{loc}(x)\to W^c_{loc}(x)
\end{equation}
defined in the following way: given $y\in W^{usc}_{loc}(x)$, there exists a unique point $p_u(y)$ in the disc $W^{sc}(x)$ that belongs to the local unstable manifold of $y$. Since $W^{sc}_{loc}(x)$ is the local stable saturation of $W^c_{loc}(x)$, then $p_u(y)\in W^{sc}_{loc}(x)$ is in the local stable manifold of a unique point $p^{us}(y)$ in $W^c_{loc}(x)$. That is, $p_{us}(y)$ is the point obtained by first projecting along unstable manifolds onto $W^{sc}_{loc}(x)$, and then projecting along stable manifolds onto $W^c_{loc}(x)$. Since the local stable and unstable foliations are continuous, $p_{su}$ is continuous.\par
\begin{figure}[h]
\includegraphics[width=6cm, bb=0 0 943 543]{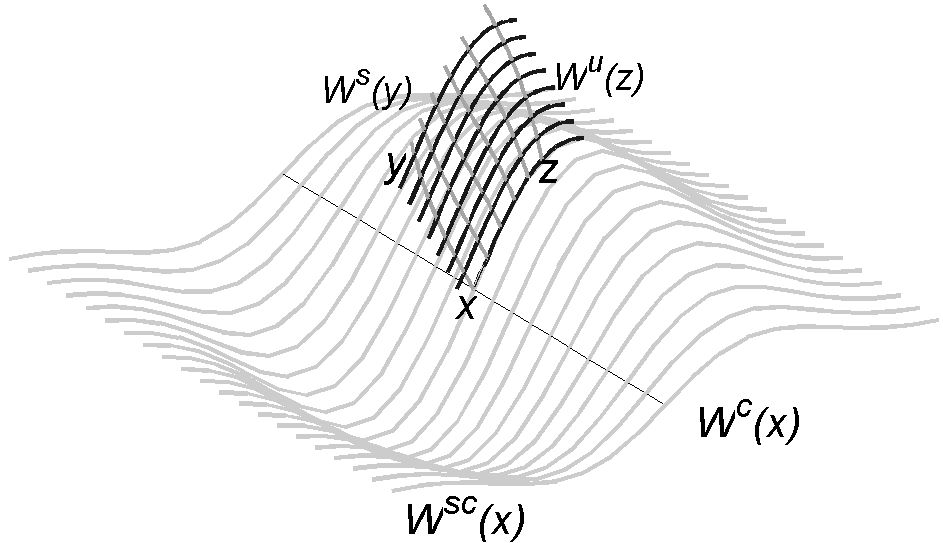}
\caption{\label{figure.closed.ac} An accessibility class in $\Gamma(f)$}
\end{figure}
Let $AC_x(y)$ be the connected component of $AC(y)\cap W^{usc}_{loc}(x)$ that contains $y$. The points of $AC_x(y)$ are the points that can be accessed by $su$-paths from $y$ without getting out from $W^{usc}_{loc}(x)$, see Figures \ref{figure.open.ac} and \ref{figure.closed.ac}. Then we have the following local description of accessibility classes of points in $\Gamma(f)$:
\begin{lemma}\label{lemma.closed.ac}
  For any $y\in W^{c}_{loc}(x)$ such that $y\in\Gamma(f)$, we have $AC_x(y)=p^{-1}_{su}(y)$
\end{lemma}
\begin{proof}
  Let $y$ be a point in $W^c_{loc}(x)$. Then $p^{-1}_{su}(y)=W^u_{loc}(W^s_{loc}(y))$, which is clearly contained in $AC_x(y)$. But also, we have $p_{su}(AC_x(y))=y$. Indeed, if $p_{su}(z)$ were different from $y$, for some $z\in AC_x(y)$, we would have a situation as described in Figure \ref{figure.open.ac}. For, since $p_{su}$ is continuous, and $AC_x(y)$ is connected, $p_{su}(AC_x(y))$ is connected. If $p_{su}(AC_x(y))$ contained another point, then it would contain a segment, which has non-empty interior in $W^c_{loc}(x)$. Proposition \ref{proposition.open.ac} then would imply that $AC(y)$ is open, which is absurd, since  $y\in\Gamma(f)$. This proves that also $AC_x(y)$ is contained in $p_{su}^{-1}(y)$.
\end{proof}
Hence, due to Lemma \ref{lemma.closed.ac} above, we have that, for each $x\in \Gamma(f)$:
$$AC_x(x)=p_{su}^{-1}(x)=W^u_{loc}(W^s_{loc}(x))\approx W^u_{loc}(x)\times W^s_{loc}(x).$$
 $W^s_{loc}(x)$ and $W^u_{loc}(x)$ are (evenly sized) embedded segments that vary continuously with respect to $x\in M$ (see Hirsch, Pugh, Shub \cite{HPS} chapters 4 and 5). This implies that $\Gamma(f)\ni x\mapsto AC_x(x)$ is a continuous map that assigns to each $x$ an evenly sized 2-disc. 
   To finish the description of accessibility classes, let us introduce the following definition:
   \begin{definition}     \label{definition.jointly.integrable} The foliations $W^s$ and $W^u$ are {\em jointly integrable} at a point $x\in M$ if there exists $\delta>0$ such that for each $z\in W^s_\delta(x)$ and $y\in W^u_\delta(x)$, we have $$W^u_{loc}(z)\cap W^s_{loc}(y)\ne\emptyset$$
   See Figure \ref{figure.closed.ac} for an illustration of a point of joint integrability of $W^s$ and $W^u$.
   \end{definition}
   Then Lemma \ref{lemma.closed.ac} and the discussion above imply the following:
   \begin{proposition}     \label{proposition.joint.integrability}
   A point $x$ belongs to $\Gamma(f)$ if and only if $W^s$ and $W^u$ are jointly integrable at all points of $AC(x)$.
   \end{proposition}
   Indeed, if $x$ belongs to $\Gamma(f)$, then for all $y\in AC(x)\subset\Gamma(f)$, we have $p_{su}(AC_y(x))=\{y\}$. In particular, if $z\in W^u_\delta(y)$ and $w\in W^s_\delta(y)$, then $W^s_{loc}(z)\cap W^u_{loc}(w)\ne\emptyset$. On the other hand, if $W^s$ and $W^u$ are jointly integrable at all points of $AC(x)$, then $AC(x)$ is a lamina, due to the explanation above (the coherence of the charts $\phi_x$ defined above depend only on the joint integrability of $W^s$ and $W^u$). Moreover, this 2-dimensional lamina is transverse to $W^c_{loc}(x)$, and so $AC(x)\cap W^c_{loc}(x)$ cannot be open. Proposition \ref{proposition.open.ac} implies $AC(x)$ is not open, so $x\in\Gamma(f)$. \par
   The following lemma shows that, in fact, the laminae of $\Gamma(f)$, that is, the accessibility classes of points in $\Gamma(f)$ are $C^1$.
   \begin{lemma}\cite[Lemma 5]{Di03}\label{lema.jointly.integrable}
     If $W^s$ and $W^u$ are jointly integrable at $x$, then the set
     $$W^{su}_{loc}(x)=\{W^u(z)\cap W^s(y): \text{with }z\in W^s_\delta(x)\text{ and } y\in W^u_\delta(x) \}$$
     where $\delta>0$ is as in the definition of joint integrability (Definition \ref{definition.jointly.integrable}), is a 2-dimensional $C^1$-disc that is everywhere tangent to $E^s\oplus E^u$.
   \end{lemma}
   In order to prove Lemma \ref{lema.jointly.integrable} we shall use the following result by Journ\'e:
   \begin{theorem}\cite{Jo88}\label{teo.journe} Let $F^h$ and $F^v$ be two transverse foliations with uniformly smooth leaves on an open set $U$. If $\eta: U\to M$ is uniformly $C^1$ along $F^h$ and $F^v$, then $\eta$ is $C^1$ on $U$.
   \end{theorem}
   \begin{proof}[Proof of Lemma \ref{lema.jointly.integrable}]
 Let $D$ be a small, smooth 2-dimensional disc containing $x$ and transverse to $E^c_x$. Consider a 1-dimensional smooth foliation of a small neighborhood $N$ of $x$, transverse to $D$. If $D$ is sufficiently small, there is a smooth map $\pi:N\to D$, which consists in projecting along this smooth 1-dimensional foliation. Note that $W^{su}_{loc}(x)$ can be seen as the graph of a continuous function $\eta:D\to N$. \par
 We produce a grid on $D$ in the following way: the horizontal lines are the projections of the stable manifolds $W^s(y)$, with $y\in W^{u}_{\delta}(x)$, that is, the horizontal lines are of the form $\pi(W^s_{loc}(\eta(v)))$, with $v\in D$. Analogously, the vertical lines are the projections of the unstable manifolds $W^u_{loc}(z)$, with $z\in W^s_\delta(x)$; that is, the vertical lines are of the form $\pi(W^u_{loc}(\eta(w)))$, with $w\in D$. \par
 Now, $v\mapsto W^s_{loc}(\eta(v))$ and $w\mapsto W^u_{loc}(\eta(w))$ are continuous in the $C^1$-topology, that is, for close $v$ we obtain close $W^s_{loc}(\eta(v))$ in the $C^1$-tolopology ($E^s$ is a continuous bundle). Since $\pi$ is smooth, we also obtain that $F^h=\{\pi(W^s_{loc}(\eta(v)))\}_{v\in D}$, the horizontal partition of $D$, and $F^v=\{W^u_{loc}(\eta(w))\}_{w\in D}$, the vertical partition of $D$, are transverse foliations continuous in the $C^1$-topology. \par
But $\eta$ is uniformly $C^1$ along $F^h$, since $\eta$ along a leaf $F^h(v_0)=\pi(W^s_{loc}(\eta(v_0)))$ is exactly $W^s_{loc}(\eta(v_0))$. Indeed, $\eta\circ\pi:W^{su}_{loc}(x)\to W^{su}_{loc}(x)$ is the identity map, and $W^s_{loc}(\eta(v_0))$ is a smooth manifold. Analogously, we obtain that $\eta$ is uniformly $C^1$ along $F^v$. Hence, by Theorem \ref{teo.journe} $\eta$ is $C^1$.
 \end{proof}

\subsection{Properties of the lamination $\Gamma(f)$ of non-open accessibility classes} In order to prove the ergodicity Conjecture \ref{conjecture.strong.non.ergodic}, a possible strategy is to accurately describe  the lamination $\Gamma(f)$. More precisely, one would like to see that non-accessibility implies the existence of a compact (toral) accessibility class.\par
However, the state of the art so far is:
\begin{theorem}(\cite[Theorem 1.6]{ParHypDim3})\label{teo.laminacion.accesible}
 Let $f : M^{3}\to M^{3}$ be a conservative partially hyperbolic diffeomorphism that is not accessible. Then one of the following possibilities
holds:
 
\begin{enumerate}
\item there is a compact accessibility class (a torus tangent to $E^{s}\oplus E^{u}$) 
\item  there exists an invariant sublamination $\Lambda\subset\Gamma(f)$ of $M$ that
trivially extends to a (not necessarily invariant) foliation without compact
leaves. Moreover, if $\Lambda\ne M$ the boundary leaves of $\Lambda$ are periodic, have Anosov
dynamics and periodic points are dense in each booundary leaf with the intrinsic topology.
\item $\Gamma(f)$ is a minimal foliation
\end{enumerate} 
\end{theorem}

 With respect to item (2), a leaf $L$ of a lamination $\Lambda$ is a {\em boundary leaf} if there is a transverse segment to $L$ containing a subsegment $\alpha$ with an endpoint in $L$ and such that $ \alpha\cap \Lambda=\emptyset$. In \cite{ParHypDim3} it is proven that boundary leaves are periodic in the conservative setting, and, moreover, that periodic points are dense in each boundary leaf with the intrinsic topology.\label{boundary.leaf}
 \par
If Case (1) holds, then Conjecture \ref{conjecture.strong.non.ergodic} is true. We conjecture that Case (2) is not possible, more precisely, we conjecture that each boundary leaf should be a torus. Answering the following question in the affirmative would rule out Case (2):

\begin{question}
Let $L$ be a complete immersed surface in a 3-manifold, such that there is an Anosov dynamics on $L$ where
\begin{enumerate}
 \item each stable and unstable manifold is complete, and angles between stable and unstable manifolds are bounded; 
 \item periodic points are dense with the intrinsic topology; and 
 \item the stable and unstable manifold of each periodic point are dense in $L$ with the intrinsic topology.
\end{enumerate}
Is $L$ the 2-torus?
\end{question}
Case (3) of Theorem \ref{teo.laminacion.accesible} means that each leaf of $\Gamma(f)$ is dense. We conjecture that in this case, in fact, $f$ is {\em essentially accessible}, this means that each set which is union of accessibility classes has either measure one or zero. Essential accessibility in dimension 3 implies ergodicity \cite{burnswilkinson}, \cite{HHU2008inv}. If this could be established, then Conjecture \ref{conjecture.strong.non.ergodic} would be proven true. Finding a counterexample, however, would be very interesting:
\begin{question} Is there an example of a partially hyperbolic diffeomorphism in a 3-manifold such that the accessibility classes of $f$ form a minimal foliation, but $f$ is not essentially accessible?
\end{question}

Since in Case (1) of Theorem \ref{teo.laminacion.accesible}, Conjecture \ref{conjecture.strong.non.ergodic} follows trivially, we would like to better describe what happens in Cases (2) and (3). The following describes the accessibility classes in these cases:

\begin{theorem}\label{thm.313} If $f$ has no compact accessibility class, then the $\pi_{1}$ of each accessibility class injects in $\pi_{1}(M)$.
\end{theorem}
\begin{proof}
The result follows almost directly from the following Theorem by Novikov:

\begin{theorem}[Novikov] \label{teo.novikov} Let $M$ be a compact orientable 3-manifold and $\mathcal{F}$ a
transversely orientable codimension-one foliation without {\em Reeb components}. Then, for each leaf $L$ in $\mathcal{F}$, $\pi_{1}(L)$ injects in $\pi_{1}(M)$
\end{theorem} 

A {\em Reeb component} of a foliation is a solid torus subfoliated by planes, as in Figure \ref{figure.reeb.comp}. \par
 
\begin{figure}
\begin{center}
 \includegraphics[width=.5\textwidth]{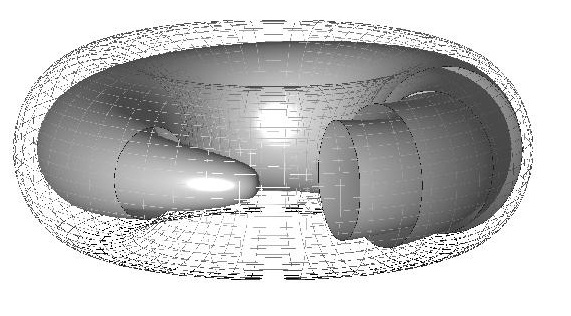}
 \caption{\label{figure.reeb.comp} Reeb component}
\end{center}
\end{figure}

\begin{question}
 Does Theorem \ref{thm.313} hold without assuming there are no compact accessibility classes?
\end{question}

If $\Gamma(f)=M$, then we are already in the hypothesis of Theorem \ref{teo.novikov}, since the fact that $\Gamma(f)$ has no compact leaves precludes the existence of Reeb components. 
The rest of the theorem follows by proving that if $\Gamma(f)\ne M$ has no compact leaves, then it can be extended to a foliation without Reeb components. This follows almost immediately from Theorem 4.1 of \cite{ParHypDim3}:
 
\begin{theorem}[Hertz-Hertz-Ures] If $\Lambda\subset\Gamma(f)$ is an orientable and transversely orientable $f$-invariant sub-lamination without compact leaves such that $\Lambda\ne M$, then all closed complementary regions of $\Lambda$ are $I$-bundles.
\end{theorem}

By taking finite coverings we may assume that $\Gamma(f)$ itself is orientable and transversely orientable. $\Gamma(f)\ne M$ has no compact leaves, therefore its complementary regions are $I$-bundles. This allows us to extend $\Gamma(f)$ to a foliation in a trivial way, by ``copying'' the boundary leaves. This means, each complementary region is of the form $L\times [0,1]$, where $L$ is a boundary leaf of $\Gamma(f)$, we foliate each complementary region, by considering leaves of the form $L\times\{t\}$, with $t\in[0,1]$. 
\end{proof}

\section{Dynamical Coherence}\label{section.dc}

It turns out that ergodicity in our setting is tightly related to integrability of the invariant bundles. As we explained before the bundles $E^s,E^u$ are always integrable. The integrability of the center bundle $E^c$, on the other hand, cannot be  always guaranteed even in our setting. This was a long standing problem and was recently solved in \cite{HHUnonDCexample}, see Theorem \ref{teo.non.dc}, rewritten below.

\espc

Let us recall {\em dynamical coherence}, which was stated in Definition \ref{definition.dc}. A partially hyperbolic diffeomorphism is dynamically coherent if there is an $f$-invariant foliation tangent to $E^{s}\oplus E^{c}$ (the {\em center-stable foliation}), and an $f$-invariant foliation tangent to $E^{c}\oplus E^{u}$ (the {\em center-unstable foliation}). 
Note that in this case the center bundle is also integrable: an $f$-invariant foliation tangent to $E^{c}$  is obtained by simply intersecting the center stable and center unstable leaves, and taking connected components. This is called the {\em center foliation}.\par

\espc

\begin{proposition}
If $f:M^{3}\rightarrow M^{3}$\ is a partially hyperbolic diffeomorphism whose center bundle is $C^1$, then $f$ is dynamically coherent.
\end{proposition}

\begin{proof}
Observe first that $\Fc$ is $f$-invariant: if $c:[0,1]\rightarrow \Wc{x}$\ is a differentiable curve with $c(0)=x$, then $f\circ c: [0,1]\rightarrow M$ is a differentiable curve tangent to $E^c$, and hence by uniquenes of solutions of differential equations, $f\circ c([0,1])\subset \Wc{f\circ c(0)}=\Wc{f(x)}$.

Theorem 6.1 and Theorem 7.6 in \cite{HPS} imply that through each leaf $L$ of  $\Fc$ there exist immersed submanifolds $\Ws{L},\Wu{L}$  tangent to $E^{cs},E^{cu}$ respectively, saturated by the corresponding strong foliations. Again using uniquenes of solutions of differential equations, one proves that the families $\Fcs=\{\Ws{L}\}_{L\in\Fc}$, $\Fcu=\{\Wu{L}\}_{L\in\Fc}$ are pairwise disjoint, and since their 
tangent spaces vary continuously, they form foliations. Invariance follows since $\Fc,\Fs,\Fu$ are invariant.
\end{proof}

More details about this can be found in \cite{burnswilkinson2008}. When the center bundle is not differentiable, we still have curves tangent to it as a consequence of Peano's Theorem. This family of curves, however, is not assembled as a foliation, but it still can contain relevant information. See \cite{HHUcenterunstable} and \cite{WeakFol}.\newline\par
\begin{problem}
Find an example  of a dynamically coherent partially hyperbolic diffeomorphism that is not leafwise conjugate to a $C^{1}$ dynamically coherent one. Can \cite{BPP} examples be adjusted to get one?
\end{problem}

Other condition that guarantees dynamical coherence in $\mathbb{T}^{3}$ is {\em absolute partial hyperbolicity}, a notion stronger than partial hyperbolicity, which was described in Equation (\ref{eq.abs.ph}):

\begin{theorem}[Brin-Burago-Ivanov -\cite{brinburagoivanov2009}] \label{teo.BBI}
If $f:\mathbb{T}^3\rightarrow \mathbb{T}^3$\ is absolutely partially hyperbolic, then $f$\ is dynamically coherent.
\end{theorem}

However, this is not the general case, as we explain in the next subsection.

\subsection{A non-dynamical coherent example.}\label{subsection.non.dc.example}

\setcounter{section}{2}
\setcounter{theorem}{12}
\begin{theorem}\cite{HHUnonDCexample} There exists a partially hyperbolic diffeomorphism $f:\T^{3}\to\T^{3}$ such that 
\begin{enumerate}
\item there is no invariant foliation tangent to the distribution $E^{c}\oplus E^{u}$; and 
\item there is an invariant 2-dimensional torus $T$  tangent to the distribution $E^{c}\oplus E^{u}$.
\end{enumerate}
 Moreover, there is a $C^{1}$-open neighborhood ${\mathcal U}$ of $f$ in $\diff^{1}(M)$ such that all $g$ in ${\mathcal U}$ satisfy conditions (1) and (2).
\end{theorem}
\setcounter{section}{4}
\setcounter{theorem}{2}

\espc

\begin{proof}[Sketch]
Let $A:\mathbb{T}^2\rightarrow \mathbb{T}^2$\ be a hyperbolic linear map with eigenvalues $\lambda<1<1/\lambda$. Take $u$ a unit eigenvector corresponding to the eigenvalue $\lambda$.  Consider also a north pole-south pole function $f: \mathbb{T}\rightarrow \mathbb{T}$\ such that 
$$
f(0)=0,f(1/2)=1/2
$$
$$
f'(1/2)=\sigma<\lambda<1<\nu=f'(0)<1/\lambda
$$
and a differentiable function $\phi: \mathbb{T}\rightarrow \Real$.

Now construct a perturbation $F$\ of the Axiom-A map $A\times f$\ by ``pushing'' in the stable direction of  $A$, namely
$$
F(x,\theta)=(Ax,f(\theta))+ (\phi(\theta)e_{s},0), \quad \phi(1/2)=0.
$$
where $e_{s}$ is a unit vector in the $E^{s}$ direction of $A$.\par
Note the strong unstable direction of $A\times f$\ is unaltered by this perturbation, and in particular the strong stable manifold of the perturbation exists and coincides with the strong stable manifold of the unperturbed map. Observe that the unperturbed map is not partially hyperbolic. Now we study the other invariant directions.

\espc

We are seeking invariant directions of the derivative of $F$:
$$
dF_{(x,\theta)}(v,t)=(Av,f'(\theta)t)+(\phi'(\theta)te_{s},0)
$$
An invariant direction (inside the $e_{s}\times \mathbb{T}$\ cylinder) will be generated by a vector field of the form $(a(\theta)e_{s},1)$\ for some function $a$, and hence we need to solve

\begin{equation}\label{coho1}
a(f(\theta))f'(\theta)=\lam a(\theta)+\phi'(\theta).
\end{equation}

We are thus led to find a solution of the cohomological equation 
$$
b\circ f=\lambda b+\phi
$$
(the solution of \eqref{coho1} is just $a=b'$). One then checks that the following two functions are solutions,
\begin{equation}\label{serie1}
\eta(\theta)=\frac{1}{\lambda}\sum_{1}^{\oo}\lam^n\phi(f^{-n}\theta)
\end{equation}
\begin{equation}\label{serie2}
\zeta(\theta)=-\frac{1}{\lambda}\sum_{0}^{\oo}\lam^{-n}\phi(f^{n}\theta)
\end{equation}
and that the previous assumptions imply that $\eta\in C^1(T\setminus\{1/2\}),\zeta\in\ C^1(T\setminus\{0\})$.

\espc
Let us define 
$$
 E^{c}(\theta)=\text{span} (\eta'(\theta)e_{s},1)\qquad \text{for}\quad \theta\ne\tfrac12
$$
and 
$$
 E^{s}(\theta)=\text{span}(\zeta'(\theta)e_{s},1)\qquad\text{for} \quad\theta\ne 0
$$

Back to the invariant directions, note that, for generic $\phi$, $\eta'(\theta)$\ gets bigger as $\theta$\ approaches $1/2$, and thus if we can choose  $\phi$\ so that 
\begin{equation}\label{nocoh1}
\lim_{\theta\rightarrow 1/2}\eta'(\theta)=\oo
\end{equation}
we will get continuity for $E^c$\ by defining
$$
E^c(\theta=\tfrac12)=\text{span}\{(e_{s},0)\}=E^s_{A}\times 0.
$$

Arguing similarly, we define
$$
E^s(\theta=0)=E^s_{A}\times 0,
$$
and we will get a continuous bundle provided that we prove 
\begin{equation}\label{nocoh2}
\lim_{\theta\rightarrow 0}\zeta'(\theta)=\oo.
\end{equation}

\espc

Assume for now that we have proved that these bundles are continuous. Now we want to show that $T\mathbb{T}^3=E^s\oplus E^c \oplus E^u$, or what is equivalent, that the angle between $E^s$\ and $E^c$\ is not zero. What we need to show is that $\eta'\neq \zeta'$\ for  $\theta\neq 0,1/2$. Note that for $\theta=0,1/2$, the angle is not zero, and hence it is not zero in a neighbourhood of these points. But by the cohomological equations,
$$
(\eta'-\zeta')\circ f=\lambda (\eta'-\zeta')
$$

\noindent and using the form of the dynamics of $f$, we conclude that the sign of $\eta'- \zeta'$\ is constant in $(0,1/2)$\ and $(1/2,1)$, and clearly non zero. The following lemma is proven in detail in \cite{HHUnonDCexample}.

\begin{lemma}
There exists $\phi:\T\to\R$\ such that: 
\begin{enumerate}
\item the limits in \eqref{nocoh1} and  \eqref{nocoh2} hold.
\item $\eta'$\ has opposite sign in $(0,1/2)$\ and $(1/2,1)$.
\end{enumerate}
\end{lemma}

In particular $F$\ is partially hyperbolic (but NOT absolutely partially hyperbolic). Finally we prove that it is not dynamically coherent. Observe that since the bundles only depend on $\theta$\ we obtain the stable, unstable and center manifolds (provided that this last one exists) by translation.

Consider the function $h:T^3\rightarrow T^2$\ given by
$$
h(x,\theta)=x-\eta(\theta)e_{s}.
$$

Then $F\circ h=h\circ A_T$\ and $h$\ is clearly surjective, hence it is a semiconjugacy. Note that we have a parametrization $l_{x}(\theta)$ of $h^{-1}(h(x,0))$ given by
$$
l_{x}(\theta)=(x,0)+(\eta(\theta)e_{s},\theta),
$$
and hence, the family of curves $\{l_{x}(\theta)\}$\ is tangent to $E^c$\ if $\theta\neq 1/2$. For $\theta\neq 1/2$\ the bundle $E^c$\ is uniquely integrable and hence its invariant curves are precisely the $l_{x}(\theta)$. But for $\theta=1/2$\ $E^c=E^s_A\times \{0\}$, hence its tangent curves have to be horizontal. Now we use that $\eta'$\ have different signs on the intervals $(0,1/2)$\ and $(1/2,1)$\ to conclude that this family is not a foliation near $\theta=1/2$, hence the bundle $E^c$\ is not integrable. See Figure \ref{nondinco}.

\begin{figure}
  \includegraphics[width=14cm]{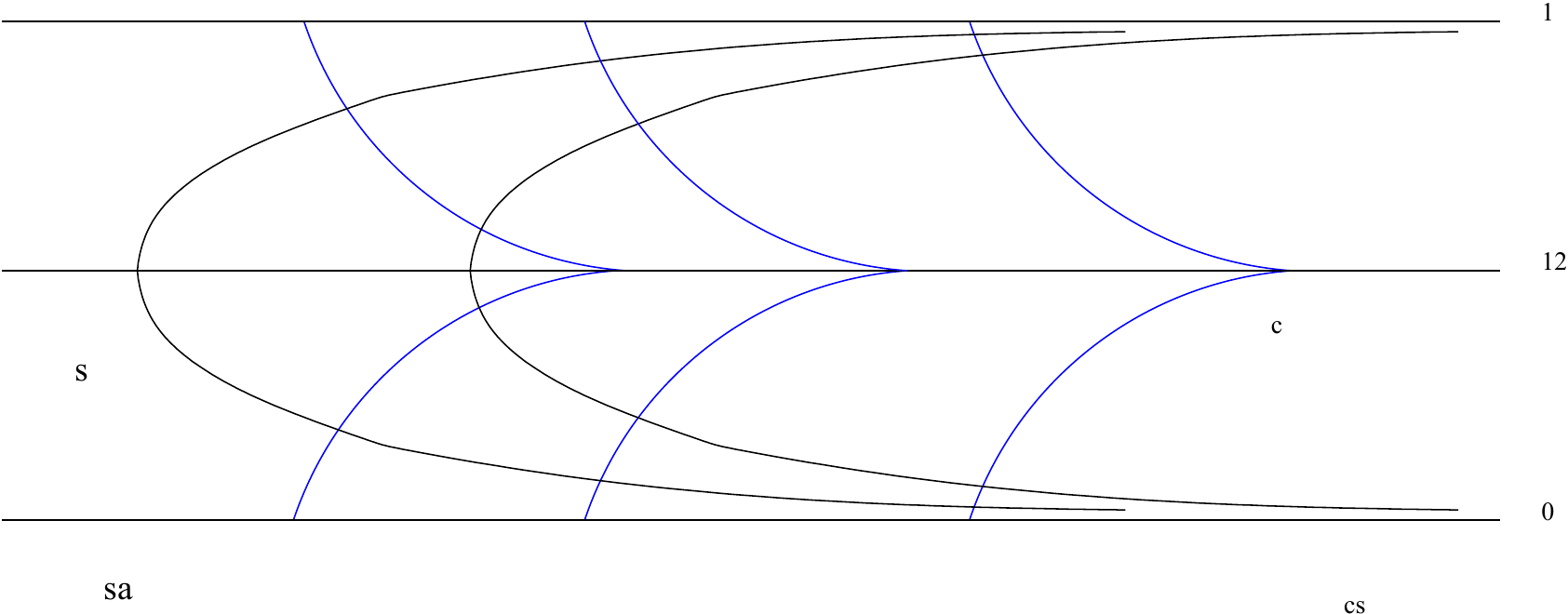}\\
  \caption{The figure above shows a center-stable leaf. The center leaves are unique for $\theta\ne 1/2$, and get tangent to the center leaf at $\theta=1/2$, making a peak. This precludes the existence of a center foliation.}\label{nondinco}
\end{figure}

\end{proof}

\espc

The example previously constructed in fact is robust, meaning that in a neighborhood of it there are no dynamically coherent partially hyperbolic diffeomorphisms, a surprising fact. This is a consequence of the fact that the invariant torus corresponding to $\theta=1/2$ is a $cu$-torus, and in particular is a normally hyperbolic torus. The results of \cite{HPS} imply that this torus persists under perturbations, meaning that any small perturbation of $f$ has a $cu$ torus. On the other hand, outside a neighborhood of this $cu$-torus, there is a unique center foliation, which is persistent under perturbations, due also to \cite{HPS}. This foliation is extended by invariance to all the complement of the $cu$-torus. Hence, if there were a foliation, it would contain a $cu$-torus. This is not possible, due to the following result:

\espc

\begin{theorem}[Hertz-Hertz-Ures] \label{teo.toro.cu} \cite{HHUnonDCexample}
Let $f:M^{3}\rightarrow M^{3}$ be a dynamically coherent partially hyperbolic diffeomorphism. Then the center unstable foliation does not have any closed leaf.
\end{theorem}

\textbf{Question:} Does the same result hold in any dimension?

\subsection{Non-dynamical coherence conjecture - state of the art}
It has been proven by A. Hammerlindl and R. Potrie \cite{hammerlindlpotrie} that the dynamical coherence Conjecture \ref{non.dc.conjecture.strong} is true in 
tori and solv- and nil-manifolds and their finite covers:

\begin{theorem}[Hammerlindl-Potrie \cite{hammerlindlpotrie} ]\label{HammerPot}
If $f$ is a non-dynamical coherent diffeomorphism in a 3-manifold with (virtually) solvable fundamental group, then there exists a 2-torus, tangent either to $E^{s}\oplus E^{c}$ or $E^{c}\oplus E^{u}$.
In particular, any partially hyperbolic diffeomorphism with $\Omega(f)=M$ in these manifolds is dynamically coherent.
\end{theorem}
\espc
This is the sharpest result concerning Conjecture \ref{non.dc.conjecture.strong} so far. 
Let us give a brief sketch of the ideas used to establish this theorem in the simpler case where $M=\mathbb{T}^{3}$. This result has been proven by R. Potrie in his thesis (see \cite{potrie}), and we shall follow his arguments. Consider  $f:\mathbb{T}^3\rightarrow \mathbb{T}^3$ partially hyperbolic, and by passing to a finite covering, it is no loss of generality to assume that the bundles $E^{\sigma}_f$ are oriented and, furthermore, that  $f$ preserves their orientation.

We will rely heavily in the  seminal papers of Brin, Burago and Ivanov (\cite{BI,BBI,brinburagoivanov2009}). The starting point is that the action in homology $f_{\ast}:H_1(\mathbb{T}^3)\rightarrow H_1(\mathbb{T}^3)$ is partially hyperbolic \cite[Theorem 1.2.]{BI}. Namely, if $A=f_{\ast}\in SL(3,\mathbb{R})$ has eigenvalues $\lambda_1,\lambda_2,\lambda_3$  then $|\lambda_1|<1<|\lambda_3|$. We have two cases:

\begin{enumerate}
\item $|\lambda_2|=1$; this is the skew-product case.
\item $|\lambda_2|\neq 1$; in this case $f$ is a DA.
\end{enumerate}

It suffices to show the existence of an $f$-invariant foliation tangent to $E^{cs}$ (the other case being analogous). It turns out that there is a  natural candidate for $\mathcal{F}^{cs}$.

\begin{theorem}[\cite{BI}, Thm. 4.1]
There exists a family $\mathcal{B}^{cs}=\{B^{cs}(x)\}_{x\in M}$ such that 
\begin{enumerate}
\item Each $B^{cs}(x)$ is an immersed boundary-less surface of class $\mathcal{C}^{1}$ tangent to $E^{cs}$
\item For every $x,x'\in M$, $x\ne x'$ the surfaces $B^{cs}(x),B^{cs}(x')$ do not cross topologically; that is, $B^{cs}(x')$ cannot intersect two different connected components of $B\setminus B^{cs}(x)$, for any neighborhood $B$ of $x$.
\item $\mathcal{B}^{cs}$ is $f$-invariant.
\end{enumerate}
\end{theorem}

The family $\mathcal{B}^{cs}$ is what is called a \emph{branched foliation}, and its elements are called leaves. To prove that $\mathcal{B}^{cs}$ is a genuine foliation it suffices to show that it is \emph{unbranched}: namely, that any $x\in M$ is contained in exactly one leaf of $\mathcal{B}^{cs}$. This fact is a direct consequence of Proposition 1.6 of \cite{Tranph} and the remark afterwards.

Another important fact is that $E^{cs}$ is {\em almost integrable}, that is, there exist a foliation (not necessarily invariant) that is transverse to $E^{u}$. This concept of almost integrability has been coined by R. Potrie and has proven useful in this context. Almost integrability of $E^{cs}$ for all 3D partially hyperbolic diffeomorphisms with orientable bundles had been established by Burago-Ivanov: 
\begin{theorem}[\cite{BI}, Key Lemma]\label{truefol3}
For every $\epsilon>0$ sufficiently small there exists  a true (not necessarily invariant) foliation $\mathcal{T}_{\epsilon}^{cs}$ 
such that the angle $\angle (T\mathcal{T}_{\epsilon}^{cs},E^{cs})$ is less than $\epsilon$. Moreover, there exists a continuous surjective map $h_{\epsilon}^{cs}:M\rightarrow M$ that is  $\epsilon$-close to the identity and sends the leaves of $\mathcal{T}_{\epsilon}^{cs}$ into leaves of $\mathcal{B}^{cs}$.
\end{theorem}

Potrie then shows that  for sufficiently small $\epsilon$, the lifted foliations $\widetilde{\mathcal{F}^u},\widetilde{\mathcal{T}_{\epsilon}^{cs}}$ to $\mathbb{R}^3$ have global product structure (that is, any two leaves $F\in \widetilde{\mathcal{F}^u}$ and $T\in \widetilde{\mathcal{T}_{\epsilon}^{cs}}$) intersect exactly in one point). This has the following consequence

\begin{theorem}[\cite{potrie} Prop. 8.4]
 $\mathcal{B}^{cs}$ is unbranched.
\end{theorem}

\espc

The proof of this theorem relies on two geometric facts:

\textbf{Fact 1:} $\widetilde{\mathcal{F}^u}$ is quasi-isometric: that is, there exist $a,b>0$ such that for every $x,y\in \mathbb{R}^3, y\in \widetilde{\mathcal{F}^u}(x)$ the length $l(x,y)$ of the interval with endpoints $x,y$ contained in  $\widetilde{\mathcal{F}^u}(x)$ satisfies
$$
l(x,y)\leq a|x-y|+b.
$$

\espc

\textbf{Fact 2:} There exists an $A$-invariant plane $P$ and $R>0$ such that every leaf of $\mathcal{B}^{cs}$ is contained in an ${R}$-neighborhood of a plane parallel to $P$, and either 
\begin{enumerate}
\item the projection of $P$ in $\mathbb{T}^3$ is dense, and the $R$-neighborhood of every leaf of $\mathcal{B}^{cs}$ contains a plane parallel to $P$, or
\item the projection of $P$ in $\mathbb{T}^3$ is a 2-torus: in this case there exists a 2-torus tangent to $E^{s}\oplus E^{c}$
\end{enumerate}

Note that if $f$ is a DA only (1) above can hold. Let us see how the proof goes:

\begin{proof}
A modification of Proposition 3.7 in \cite{BBI} gives that a codimension one foliation in $\mathbb{T}^3$ either has a Reeb component, or there exists $R>0$ and an $A$-invariant plane $P\subset \mathbb{R}^3$ such that every leaf of the lifted foliation is in a $R$-neighborhood of $P$. The projection of a plane in $\mathbb{T}^3$ is either dense or a 2-torus: in the former case every leaf of the lifted foliation is parallel to a fixed translate of the plane, while in the later case there is a leaf of the foliation in $\mathbb{T}^3$ which is a torus (see Theorem 5.3 and Proposition 5.6 in \cite{potrie}).

As  $\mathcal{T}_{\epsilon}^{cs}$ is transverse to $\Fu$, it cannot have a Reeb component and thus we can consider the plane $P$ as above. Note that by \ref{truefol3} the leaves of $\widetilde{\mathcal{B}^{cs}}$ are $\epsilon$-close to the leaves of $\widetilde{\mathcal{T}_{\epsilon}^{cs}}$, hence the first part of the claim follows.

Now, it suffices to observe that if $\mathcal{T}_{\epsilon}^{cs}$ has a torus leaf; by Theorem \ref{truefol3} $\mathcal{B}^{cs}$ also has a torus leaf, therefore, there is a 2-torus tangent to $E^{s}\oplus E^{c}$. 
\end{proof}

Once the above facts are established, the proof of the theorem is carried by standard arguments. To finish the proof of the conjecture, one has to analyze the skew-product and the DA case separately. In both cases, with the machinery developed, it is not too hard to check that the absence of tori tangent to $E^{s}\oplus E^{c}$ implies the global product structure referred above, thus implying dynamical coherence.

\espc

When the manifold $M$ is a solv-manifold, but not a nilmanifold, the proof of Theorem \ref{HammerPot} becomes technically harder, although some of the guidelines presented are still valid. It relies on more background on foliation theory (codimension-one foliations of compact solv-manifolds are reasonably well understood). Solv-manifolds of this type are covered by torus bundles over the circle (a reasonable geometric object), and then there is a canonical model (isotopic to the identity) to compare the dynamics. For more details an a complete proof, we refer the reader to \cite{hammerlindlpotrie}.\par
We also recommend the excellent survey by Hammerlindl and Potrie \cite{hammerlindl.potrie.survey}.

\espc

\section{Classification}\label{classification}
For many years, the only known examples of partially hyperbolic diffeomorphisms in 3-manifolds were the ones listed in Subsection \ref{subsection.examples}, namely: time-one maps of Anosov flows, skew products, DA-diffeomorphisms, and their perturbations. As it was stated in Subsection \ref{subsection.classification}, this led E. Pujals, in 2001 to conjecture that for transitive ones, this was the complete list of partially hyperbolic diffeomorphisms:

\espc

\setcounter{section}{2}
\setcounter{theorem}{19}
\begin{conjecture}[Pujals (2001)] Any transitive partially hyperbolic diffeomorphism in a 3-manifold is finitely covered by a map which is conjugated either to
\begin{enumerate}
 \item A perturbation of the time-one map of an Anosov flow.
 \item A perturbation of an Skew Product.
 \item A DA.
\end{enumerate}
\end{conjecture}
\setcounter{section}{5}
\setcounter{theorem}{0}

Two different particular cases
of this conjecture were verified in the transitive setting by C. Bonatti and A. Wilkinson \cite{Tranph}.

\espc

\begin{theorem}[Bonatti-Wilkinson]
Let $f:M\rightarrow M$ be a transitive partially hyperbolic diffeomorphism. 
\begin{enumerate}
 \item Assume that there exists some embedded circle $c$ such that\footnote{Interestingly enough, the proof does not extend to the case 
where $c$ is merely periodic.} $f(c)=c$, with the property that for some $\epsilon>o$ the set 
$$
\bigcup_{x\in c}\Wsl{x}\cap\bigcup_{y\in c}\Wul{y}\setminus c 
$$
contains a connected component that is a circle. Then $f$ is dynamically coherent and finitely covered by a map which is conjugated to a circle extension of
an Anosov map (a topological Skew Product).

\item Assume that $f$ is dynamically coherent, and that for some $\epsilon>0$ each end of a center leaf contained in
$$
\bigcup_{x\in c}\Wsl{x}
$$
is periodic. Then, the center foliation is fixed under $f^{n}$ and it supports a continuous flow
conjugate to an expansive transitive flow.

\end{enumerate}
\end{theorem}

It would be interesting to know if in case (2) one in fact can take an Anosov flow, and thus settle Pujals's conjecture for that case. 
This still remains an open problem.

\espc

More recently, a new type of non-dynamically coherent example was presented, the one described in Subsection \ref{subsection.non.dc.example}. The examples in Subsection \ref{subsection.non.dc.example} suggested another possibility:

\espc
\setcounter{section}{2}\setcounter{theorem}{21}

\begin{conjecture}[{\bf Classification conjecture:} Hertz-Hertz-Ures (2009)]  Let $f$ be a partially hyperbolic diffeomorphism of a 3-manifold. \par
If $f$ is dynamically coherent, then it is (finitely covered by) one of the following:
\begin{enumerate}
\item a perturbation of a time-one map of an Anosov flow, in which case it is leafwise conjugate to an Anosov flow;
\item a skew-product, in which case it is leafwise conjugate to a skew-product with linear base; or
\item a DA, in which case it is leafwise conjugate to an Anosov diffeomorphism of $\T^{3}$.
\end{enumerate}
 If $f$ is not dynamically coherent, then there are a finite number of 2-tori tangent either to $E^{c}\oplus E^{u}$ or to $E^{s}\oplus E^{c}$, and the rest of the dynamics are trivial, as in the non dynamically coherent example \cite{HHUnonDCexample}.
\end{conjecture}

\espc\setcounter{section}{5}\setcounter{theorem}{1}

Both Conjectures have been proven false very recently by Bonatti, Gogolev, Parwani and Potrie, see \cite{BPP, BGP} and Subsection \ref{anomalous} for a descrption of these examples. However, both conjectures are true in certain manifolds, as it was proven by Hammerlindl and Potrie:

\begin{theorem}[Hammerlindl-Potrie \cite{hammerlindlpotrie}]\label{theorem.hammerlindl.potrie}
Both Conjecture \ref{conjecture.pujals} and Conjecture \ref{hhu.conj.classification} are true on 3-manifolds with (virtually) solvable fundamental group. 
\end{theorem}

Theorem \ref{theorem.hammerlindl.potrie} was first proved in tori by Hammerlindl in his thesis \cite{hammerlindl}, it was later extended to 3-manifolds with (virtualy) nilpotent groups by Hammerlindl and Potrie in \cite{PointNil}. Finally, it was extended to 3-manifolds with (virtually) solvable groups, by the same authors in \cite{hammerlindlpotrie}, still in press. \par

In \cite{hammerlindlpotrie} it is proven that, for solvmanifolds, as stated in Conjecture \ref{non.dc.conjecture.strong}, the absence of tori tangent to either $E^{s}\oplus E^{c}$ or $E^{c}\oplus E^{u}$ implies dynamical coherence. Observe that the existence of such a torus implies the existence of either a repelling or an attracting periodic torus (see more details in Section \ref{section.anosov.tori}). Transitivity precludes this possibility. Therefore, for solvmanifolds, we can assume there is dynamical coherence in both Conjectures \ref{conjecture.pujals} and \ref{hhu.conj.classification}.

Let us give a flavor of how Theorem \ref{theorem.hammerlindl.potrie} is proved in the case of solvmanifolds with non-virtually nilpotent fundamental group. \par
\begin{theorem}[\cite{hammerlindlpotrie}] \label{teorema} If $f:M\to M$ is a dynamically coherent  partially hyperbolic diffeomorphism on a solvmanifold whose fundamental group is not virtually nilpotent, then a finite cover of a finite iterate of $f$ is leafwise conjugate to the time-one map of a suspension Anosov flow. 
\end{theorem}
See Definition \ref{def.leaf.conjugacy} for the definition of leaf conjugacy. \par
To start the proof of Theorem \ref{teorema}, one shows that any such manifold is finitely covered by the mapping torus $M_{A}$ of a hyperbolic automorphism on $\T^{2}$, that is $M_{A}=\T^{2}\times {\mathbb R}/\sim$ such that $(Ax,t)\sim (x,t+1)$, where $A$ is a hyperbolic automorphism of $\T^{2}$. And any diffeomorphism of $M_{A}$ has a finite iterate that is homotopic to the identity. This is not hard to prove.\par
Now, on the universal cover of $M_{A}$, there are model foliations $\A^{cs}$ and $\A^{cu}$:
$(v_{1},t_{1})$ and $(v_{2},t_{2})$ belong to the same leaf of the foliation $\A^{cs}$ if and only if $v_{1}-v_{2}$ is in the stable eigenspace of the automorphism $A$. Similarly, $(v_{1}, t_{1})$ and $(v_{2},t_{2})$ belong to the same leaf of the foliation $\A^{cu}$ if and only if $v_{1}-v_{2}$ is in the unstable eigenspace of the automorphism $A$. In \cite{hammerlindlpotrie} it is seen that the lift to the universal cover of any foliation without compact leaves is {\em almost parallel} to either $\A^{cs}$ or $\A^{cu}$. Two foliations ${\mathcal F}$ and ${\mathcal F}'$ are {\em almost parallel} if there is a uniform bound $R>0$ such that 
\begin{enumerate}
 \item for each leaf $L\in {\mathcal F}$ there is a leaf $L'\in{\mathcal F}'$ such that $d_{H}(L,L')<R$, and
  \item for each leaf $L'\in {\mathcal F}'$ there is a leaf $L\in{\mathcal F}$ such that $d_{H}(L,L')<R$,
\end{enumerate}
 where $d_{H}$ is the Hausdorff distance, that is 
 $$d_{H}(L,L')=\max\left\{\sup_{x\in L}d(x,L'), \sup_{y\in L'}d(y,L)\right\}.$$
 Now, neither ${\mathcal F}^{cs}$ nor ${\mathcal F}^{cu}$ contain compact leaves \cite{HHUcenterunstable}, and therefore each one is almost parallel to either $\A^{cs}$ or $\A^{cu}$. They proceed then to show that if ${\mathcal F}^{cs}$ is almost parallel to $\A^{cs}$, then ${\mathcal F}^{cu}$ is almost parallel to $\A^{cu}$. This step is more delicate. \par
%
Note that the center leaves of the model foliation - that is, the leaves in $\A^{c}$ that are intersection of leaves $\A^{sc}$ and $\A^{cu}$ - correspond to trajectories of an Anosov flow, which is {\em infinitely expansive}. Infinite expansivity means that for any two different points $x$ and $y$ in the universal cover and any $K>0$, there will be a time $t\in \R$ such that $X_{t}(x)$ and $X_{t}(y)$ are $K$-apart. Therefore, any two such leaves $A_{1}^{c}$ and $A_{2}^{c}$ are at infinite Hausdorff distance. This implies that the almost-parallel relation defined above assigns to each center leaf $F^{c}$ in the intersection of $\F^{sc}$ and $\F^{cu}$ a unique center leaf in $\A^{c}$. Less trivially, there is a unique leaf in $\F^{sc}$ at finite Hausdorff distance of each leaf in $\A^{sc}$, and a unique leaf in $\F^{cu}$ at finite distance of each leaf in $\A^{cu}$ (Lemma 5.3 of \cite{hammerlindlpotrie}).
Therefore, any two center leaves $F^{c}_{1}$ and $F^{c}_{2}$ that are at finite Hausdorff distance of each other, must be in the intersection of a single leaf of $\F^{sc}$ and a single leaf of $\F^{cu}$. 
\par 
 Now, let us assume that the center bundle $E_{f}^{c}$ is orientable, for otherwise we can take a finite cover. Then there exists a field $X^{c}$ tangent to $E^{c}$, defining a flow $\varphi$ on $M_{A}$. We claim that $\varphi$ is an expansive flow. \par
Indeed, any two $\varphi$-trajectories that at most $\varepsilon$-apart correspond to two center leaves that are at finite Hausdorff distance; hence, they are in the intersection of a single leaf of $\F^{sc}$ and a single leaf of $\F^{cu}$. This implies either that a stable leaf intersects (at least) twice a center unstable leaf of $\F^{cu}$ or that an unstable leaf intersects (at least) twice a center stable leaf of $\F^{sc}$. A classical argument \`a la Novikov, implies the existence of a compact leaf either in $\F^{cu}$ or in $\F^{sc}$, a situation precluded by \cite{HHUcenterunstable}. \par
Finally, M. Brunella establishes in \cite{brunella} that any expansive flow on a torus bundle is leafwise conjugate to a transitive Anosov suspension, concluding the classification theorem in solvmanifolds. \par
Again, we refer the reader to the survey of Hammerlindl and Potrie, in order to deepen in this topic.

\subsection{Anomalous partially hyperbolic diffeomorphisms}\label{anomalous}

Very recently, C. Bonatti, A. Gogolev,  K. Parwani and R. Potrie  found both a dynamically coherent example and a transitive example that are not leaf-wise conjugate to any of the above models, proving both conjectures wrong \cite{BPP, BGP}. 

The idea behind both constructions is to perform a large ``perturbation" by composing  the time one map of certain Anosov flows in an appropriate neighborhood with a map of the form $(g,Dg)$, where $g$ is a Dehn twist and $Dg$ is the derivative acting on the unitary tangent bundle. This neighborhood has to be large enough in order to obtain that the effect of the derivative of the Dehn twist be negligible. We will explain this with some more detail for the case of the perturbations of the time-one map of the geodesic flow of surfaces of constant negative curvature in Theorem \ref{BGP}.

The first family of (non-transitive) examples is obtained by modifying the Franks-Williams' construction \cite{FW} of a non-transitive Anosov flow. 

\begin{theorem} [Bonatti, Parwani, Potrie, \cite{BPP}] There is a closed orientable 3-manifold $M$ endowed with a non-transitive
Anosov flow $X_{t}$ and a diffeomorphism $f : M\rightarrow M$ such that:
\begin{itemize}
\item $f$ is absolutely partially hyperbolic,
\item $f$ is robustly dynamically coherent,
\item the restriction of $f$ to its chain recurrent set coincides with the time-one map of
the Anosov flow $X_{t}$, and
\item for any $n > 0$, $f^n$ is not isotopic to the identity.
\end{itemize}
\end{theorem}

The transitive examples are built on time-one maps of two different Anosov flows: the Bonatti-Langevin example \cite{BL} and the geodesic flow of surfaces of negative constant curvature (a similar construction works for the Handel-Thurston Anosov flow 	\cite{HT}). After the statement of the theorem we will give a brief outline of the construction for the case of geodesic flows. 
 
\begin{theorem}[Bonatti, Gogolev, Potrie, \cite{BGP}] \label{BGP}There exist a closed orientable 3-manifold $M$ and an absolutely
partially hyperbolic diffeomorphism $f : M \to M$ that satisfy the
following properties
\begin{itemize}
\item $M$ admits an Anosov flow;
\item $f^n$ is not homotopic to the identity map for all $n > 0$;
\item f is volume preserving; and
\item f is robustly transitive and stably ergodic.
\end{itemize}
\end{theorem}

\noindent{\bf Unit tangent bundle of surfaces of negative curvature:} Here we give the main ideas of the construction of the examples for this case. Let $S$ be a surface and $g$ a Riemannian metric of curvature $-1$. Fix a simple closed geodesic $\gamma$. It is possible to deform the hyperbolic metric in such a way the length of $\gamma$ goes to $0$. Indeed, there is a sequence $g_n$ of metrics of curvature $-1$ such that $\text{length}_{g_n}\gamma\to 0$. Then there are collars of uniform length (for the metric $g_n$)  of $\gamma$; call them $C_n$. These collars become thinner and thinner as $n$ goes to infinity. This implies that Dehn twists $\rho_n$ on these collars are very close to isometries, for $n$ large enough. \par
Now consider $h_n=D_{\rho_n}\circ \varphi_n$ where $\varphi_n$ is the time-one map of the geodesic flow of $g_n$ and $D_{\rho_n}$ is the projectivization of the derivative of $\rho$. Since the  partially hyperbolic structure does not change with the metric  (indeed the partially hyperbolic structure depends on the metric, and for all $n$ the metric on the universal cover is the same) this will imply that  $h_n$ is partially hyperbolic for 
$n$ large enough. By constructing $\rho_n$ with some care, $D_{\rho_n}$ can be made volume preserving. Known results and techniques imply that there is a stably ergodic and robustly transitive perturbation of $h_n$. 

Bonatti, Hammerlindl, Gogolev and Potrie have announced a generalization of the latter construction, see \cite{BGP}. There is a natural homomorphism of the mapping class group of a surface of genus greater than one, $S$, into the mapping class group of its unit tangent bundle given by the projectivization, $I:\text{MCG}(S)\to \text{MCG}(T_1S)$. 

\begin{theorem} Each mapping class of the image of $I$ admits a volume preserving partially hyperbolic representative. 
\end{theorem}

There are some open questions about the  examples given by Theorem \ref{BGP}. The most important is if they are dynamically coherent. 

Regarding the classification many new questions arise. Some of them are the following. 

\begin{question} Suppose the fundamental group of $M$ is not (virtually) solvable. If $M$ admits a partially hyperbolic diffeomorphism, does it support an Anosov flow? 
\end{question} 

Hammerlindl, Potrie and Shannon have announced that the answer is positive for Seifert manifolds having fundamental group with exponential growth. 

\begin{question} Given a partially hyperbolic diffeomorphism on $M$, is there some sort of inverse process of the previous construction leading to a partially hyperbolic diffeomorphism isotopic to the identity? To a partially hyperbolic diffeomorphism leaf conjugate (up to finite cover and iterate) to the time one map of an Anosov flow?
\end{question}

\section{A tool to better understand some partially hyperbolic dynamics: Anosov Tori}\label{section.anosov.tori}

Both the ergodicity and the integrability conjectures propose that the existence of a map with some specific dynamical property leads to very rigid restrictions in the topology of the ambient manifold. The reader may wonder why this is case, and how one can attempt to prove such type of results. We discuss these issues in this section.

\espc

The unifying link is, surprisingly, the existence of certain tori embedded in the manifold.

\begin{definition}
We say that the manifold $M$ admits an {\em Anosov torus} if there exist a $C^{1}$-embedded torus $T\subset M$ and a diffeomorphism $\phi:M\rightarrow M$ such that
\begin{enumerate}
\item $\phi(T)=T$.
\item $\phi|T$ is a linear hyperbolic automorphism\footnote{This is equivalent to the existence of $\phi$ such that $\phi|T$ be isotopic to an Anosov diffeomorphism, which holds if and only if the action on the first homology group of the torus $H_1(T)$ is hyperbolic.}.
\end{enumerate}
\end{definition}

As we shall see below, not every manifold admits an Anosov torus. 
\espc
 
Recall that a three manifold is irreducible if every embedded two-sphere bounds a three-ball. We then have the following topological result:
 
\begin{theorem}[Hertz-Hertz-Ures \cite{AnosovTori}]\label{clasifica3}
Assume that $M$\ is a compact, irreducible 3-manifold supporting an Anosov torus. Then $M$ is homeomorphic to either:
\begin{enumerate}
\item a $3$-torus,
\item the mapping torus of $-Id:\mathbb{T}^2\rightarrow\mathbb{T}^2$, or
\item the mapping torus of an hyperbolic automorphism of $\mathbb{T}^2$.
\end{enumerate}
\end{theorem}

\espc
 We remark that partial hyperbolicity is not required in Theorem \ref{clasifica3}. This theorem just shows that the 3-manifolds admitting Anosov tori are actually very few. 
Now we apply this result to the partially hyperbolic context. We first note that irreducibility comes for free in this setting.

\begin{lemma}
If a 3-manifold $M$ supports a partially hyperbolic diffeomorphism, then $M$ is irreducible.
\end{lemma}

\begin{proof}
A 3-manifold admitting a partially hyperbolic diffeomorphism has a
codimension-one foliation having neither Reeb components nor spherical leaves
\cite{BBI}. This proves the claim, since Rosenberg shows in \cite{FolPlanes} that any codimension-one
foliation in a reducible 3-manifold must have a Reeb component or a spherical
leaf. See also \cite{FeuDimTrois}.
\end{proof} 

\espc

The following proposition describes Anosov tori that arise naturally in partially hyperbolic dynamics, see \cite{HHUcenterunstable} for more details.

\begin{proposition}\label{exAnTor}
Let $f:M\rightarrow M$ be a partially hyperbolic diffeomorphism, and assume that there exists an $f$-invariant embedded torus $T$ tangent to either $E^{c}\oplus E^{u}$, $E^{s}\oplus E^{c}$ or $E^u\oplus E^{s}$. Then $T$ is an Anosov torus.
\end{proposition}

\begin{proof}
Let $g=f|T$. In each of the different cases, $g$ or $g^{-1}$ preserves an expanding foliation by lines, so with no loss of generality we will assume that $\norm{dg|E^u}>1$. It suffices to prove that $g_{\ast}:\pi_1(T)\approx \mathbb{Z}^2 \rightarrow \mathbb{Z}^2$ is hyperbolic.

By taking $g^2$ if necessary we can suppose that $g$ preserves the orientation of $E^u|T$. Since $g$ preserves a foliation without compact leaves, the integral matrix $g_{\ast}$ 
has an eigenspace of irrational slope. This implies that either $g_{\ast}$ is hyperbolic or $g_{\ast}=Id$. In the second case $g$ has a lift  $\widehat{g} : \mathbb{R}^2 \rightarrow \mathbb{R}^2$ such that $\widehat{g}= Id +\alpha$ where $\alpha$ is a periodic, and in particular bounded, function. Hence there exists a constant $K > 0$ such that given any $X\subset \mathbb{R}^2$,
$$ \text{diam}(\widehat{g}^n(X)) \leq \text{diam}(X)+nK.$$ 

Let $\gamma$ be an arc contained in a leaf of $\Wu{x}, x\in T$. Then the length of $\gamma$ grows exponentially under $\hat g^{n}$ and its diameter grows at most linearly. This implies that given a small $\epsilon > 0$ there exists an iterate of $g$ that contains a curve of length arbitrarily large and with end points at distance less than $\epsilon$. Using Poincar\'e-Bendixon we obtain a singularity of the foliation $W^{u}$. This is a
contradiction; thus $g_{\ast}$ is hyperbolic.

\end{proof}

\espc

\begin{definition}
Let $f:M\rightarrow M$ be a partially hyperbolic diffeomorphism. An embedded torus tangent to either $E^{cs}, E^{cu}$ or $ E^{su}$ will be called a $cs,cu$ or $su$-torus respectively. In these cases, we say that $f$ \emph{admits} the corresponding torus.
\end{definition}

\begin{lemma}\label{su.torus}
If $f$ admits an $su$-torus, then $M$ admits an Anosov torus
\end{lemma}
\begin{proof}
Assume $f$ admits an $su$-torus, and consider the lamination $\Lambda$ of all $su$-tori of $f$. This is a compact lamination \cite{Hae62}. Therefore, there is a recurrent leaf., that is, there is a torus $T$ and an iterate $n$, such that $d_{H}(f^{n}(T),T)<\eps$ for small $\eps$. There exists a diffeotopy $i_{t}$ on $M$, taking $f^{n}(T)$ into $T$. Then $\phi=f^{n}\circ i_{1}$ fixes $T$ and $\phi|T$ is isotopic to an Anosov diffeomorphism.  
\end{proof}

\begin{lemma}\label{sc.cu.torus}
If $f$ admits an $sc$ or $cu$ torus, then it admits an $f$-periodic  $sc$ or $cu$ torus. 
 Therefore, $M$ admits an Anosov torus (by Proposition \ref{exAnTor})
\end{lemma}
\begin{proof}
 Let $T$ be a $cu$-torus, and consider the sequence $f^{-n}(T)$. Since the family of all compact subsets of $M$ considered with the Hausdorff metric $d_{H}$ is compact, there is a subsequence $f^{-n_{k}}(T)$ converging to a compact set $K\subset M$. Therefore, for each $\eps>0$ there are arbitrarily large $N>>L>0$ such that $d_{H}(f^{-N}(T), f^{-L}(T))<\eps$. \par
 Since $T$ is transverse to the stable foliation, the union of all local stable leaves of $T$ forms a small tubular neighborhood of $T$, $U(T)$. Since stable leaves grow exponentially under $f^{-1}$, if $N>>L$ as above are large enough, then $f^{-L}(\overline{U(T)})\subset f^{-N}(U(T))$. This implies that 
 $f^{N-L}(\overline{U(T)})\subset U(T)$. 
\begin{exercise}
 Finish the proof by showing that $\cap_{k=0}^{\infty}f^{k(N-L)}(U(T))$ is a periodic $cu$-torus. 
\end{exercise}
\end{proof}
\espc

\begin{corollary}
Suppose that $f:M\rightarrow M$ is a partially hyperbolic diffeomorphism admitting an $sc,cu$ or $su$ torus. If $M$ is connected, then $M$ is homeomorphic to either
\begin{enumerate}
\item A $3$-torus.
\item The mapping torus of $-Id:\mathbb{T}^2\rightarrow\mathbb{T}^2$.
\item The mapping torus of an hyperbolic automorphism of $\mathbb{T}^2$.
\end{enumerate}
\end{corollary}

Observe that an $sc$ or $cu$ torus cannot appear on the conservative setting. Indeed, by Lemma \ref{sc.cu.torus} above it would imply the existence of a {\em periodic} $sc$ or $cu$ torus. This 2-torus is, respectively, repelling or attracting, a situation that cannot occur in a conservative setting.  \par
In the following subsection we sketch the proof of Theorem  \ref{clasifica3}. We refer the reader to \cite{AnosovTori} for the complete proof.


\espc

\subsection{Manifolds admitting Anosov tori}

The first step in the proof of Theorem \ref{clasifica3} is to show that Anosov tori are incompressible. 

\begin{definition}
An embedded orientable surface $S\subset M$ is \emph{incompressible} if the homomorphism
induced by the inclusion map $i_{\ast} : \pi_1(S)\rightarrow \pi_1(M)$ is injective.
\end{definition}

Equivalently, $S$ is incompressible if every embedded disk $D^2\subset M$ such that $D^2\cap S=\partial D^2$, is contractible 
in $S$ (see for instance \cite{Hatcher}, page 10). We have\par

\begin{theorem}[Hertz-Hertz-Ures \cite{ParHypDim3}]
Anosov tori are incompressible.
\end{theorem}

Now, let us assume, as in the hypotheses of Theorem \ref{clasifica3}, that the irreducible 3-manifold $M$ admits an Anosov torus $T$.  Since $T$ is incompressible, we can ``cut'' $M$ along $T$ and obtain a manifold $N$ having  incompressible $2$-tori as boundary components. Theorem \ref{clasifica3} then follows from the following theorem:

\espc

\begin{theorem}[Hertz-Hertz-Ures \cite{AnosovTori}]\label{reduced}
Let $N$ be a compact orientable irreducible $3$-manifold with nonempty
boundary such that all the boundary components are incompressible $2$-tori. If $N$ admits an Anosov torus, then $N \approx  \mathbb{T}^2\times [0,1]$.
\end{theorem}

In order to prove this, we make use of the Jaco-Shalen-Johannson decomposition, or JSJ-decomposition, which states that any manifold satisfying the hypotheses of Theorem \ref{reduced} can be cut by a (unique) family of incompressible tori, so that the remaining pieces have certain characteristics: they are either {\em Seifert} or else {\em atoroidal} and {\em acylindrical}. We provide these definitions below. See also Theorem \ref{jsj.decomposition}.\par
The proof of Theorem \ref{reduced} consists in showing, on one hand, that any Seifert manifold having incompressible tori as boundary components is $\T^{2}\times[0,1]$, and, on the other hand, that any manifold satisfying the hypotheses of Theorem \ref{reduced} that is atoroidal has an annulus which is properly embedded and is not isotopic to the boundary of the manifold. This last statement contradicts that the manifold is {\em acylindrical}, and shows that every component in the JSJ-decomposition must be Seifert, which proves the theorem.\par 

Any compact 3-manifold, with or without boundary, supporting a foliation by circles is a {\em Seifert manifold} (see \cite{PointPerHomeo}). This was not the original definition, a more descriptive one is the following:

\begin{definition}
A \emph{Seifert manifold} is a 3-manifold that admits a decomposition into disjoint circles,
the fibers, such that each fiber has a neighborhood diffeomorphic, preserving
fibers, to either
\begin{enumerate}
\item A solid torus foliated by horizontal circles.
\item A solid torus foliated by the fibration obtained by the identification $D^2\times [0,1]/x\sim R_{p/q}(x)$, where $R_{p/q}:D^2\rightarrow D^2$\ denotes the rotation of angle
$p/q$, and $p,q$ are coprime.
\end{enumerate}
If the manifold has boundary, its connected components are required to be tori, which are also fibered by circles.  
\end{definition}

The circles of the first type are the \emph{generic fibers}, while the ones of the second type are the \emph{singular fibers}. For an introduction to Seifert spaces see \cite{BrinSeifert}.

\espc

\begin{definition}
Let $N$ be a 3-manifold with boundary.
\begin{enumerate}
\item $N$ is \emph{atoroidal} if every incompressible torus is
$\partial$-\emph{parallel}, that is, isotopic to a subsurface of $\partial N$.
\item $N$ is \emph{acylindrical} if every incompressible annulus $A$ that is properly embedded
(i.e., $\partial A \subset \partial N$) is $\partial$-parallel by an isotopy fixing $\partial A$.
\end{enumerate}
\end{definition}

As we mentioned above, any irreducible orientable 3-manifold having incompressible tori as boundary components admits a natural decomposition into Seifert
pieces on one side, and atoroidal and acylindrical components on the other.

\espc

\begin{theorem}[JSJ-decomposition - \cite{Hatcher}]\label{jsj.decomposition}
If $N$ is an irreducible, orientable
$3$-manifold, having incompressible tori as boundary components, then there exists a finite collection $\mathcal{T}$ of disjoint incompressible tori
 such that for each component $N_i$ of $N\setminus \bigcup \mathcal{T}$, either
\begin{enumerate}
\item $N_i$ is a Seifert manifold, or
\item $N_i$ is both atoroidal and acylindrical.
\end{enumerate}

Any minimal such collection is unique up to isotopy. This means that if $\mathcal{T}$ is a
collection as described above, it contains a minimal subcollection $m(\mathcal{T})$ satisfying
the same claim. All collections $m(\mathcal{T})$ are isotopic.
\end{theorem}

Any minimal family of incompressible tori as described above is called a \emph{JSJ-decomposition}
of $N$. Note that if $N$ is either atoroidal and acylindrical or Seifert, then $\mathcal{T} =\emptyset$. 

\espc

Let us sketch how Theorem \ref{reduced} is proved for the case of Seifert manifolds; the other case is more delicate, and we refer the reader to \cite{AnosovTori} for the complete proof. \par

Assume $N$ is a Siefert manifold, so it admits a foliation by circles, which is called a Seifert fibration. We lose no generality in assuming that one of the incompressible tori of the boundary of $N$ is an Anosov torus $T$. By the definition of Anosov torus, there exists a diffeomorphism $\phi$ on $N$ such that it is a linear hyperbolic automorphism of $T$. The image of the Seifert fibration by $\phi$ is another Seifert fibration, which is non-isotopic to the original one on $T$. \par
But orientable manifolds admitting two Seifert fibrations that are non-isotopic on its boundary are completely classified: 

\begin{lemma}\cite{Hatcher}
If $N$ admits two Seifert fibrations that are non-isotopic on $\partial N$, then $N$ is homeomorphic to either:
\begin{enumerate}
\item the solid torus,
\item a twisted $I$ bundle over the Klein bottle, or
\item the torus cross the interval.
\end{enumerate}
\end{lemma}

In the first two cases $\partial N$ consists of a single torus, while in last one it consist of two disjoint tori. To finish the proof, it suffices then to discard the first two cases:

\espc

\begin{lemma}
If $\partial N$ contains an Anosov torus, then it contains more than one.
\end{lemma} 

\begin{proof}
Assume that $\partial N$ is a torus $T$, and consider the inclusion map $i: H_1(\partial N)\rightarrow H_1(N)$. Let $\ker{i}$ be the
kernel of the map. Then by Lemma 3.5 in \cite{Hatcher}, we have 
$$
rank(\ker{i})=\frac1{2}rank(H_1(\partial N))
$$

where $rank$ denotes the number of $\mathbb{Z}$ summands in a direct sum splitting into cyclic groups. If $\partial \tilde M=T$, then $\frac1{2}rank(H_1(T))=1$, and hence $K= 
\ker{i}$ is a one-dimensional subspace of $H_1(T)$. We have then  that $f_{\ast}(K) = K$,
where $f_{\ast}: H_1(T)\rightarrow H_1(T)$ is the isomorphism induced by any diffeomorphism
$f : N\rightarrow N$. This implies that $f_{\ast}$ has $1$ as an eigenvalue. Hence, $f$ cannot
be isotopic to a hyperbolic automorphism of $T$ . This implies that $T$ cannot be
an Anosov torus.
\end{proof}


\bibliographystyle{alpha}
\bibliography{biblio}
\end{document}